\documentclass[a4paper,amsart,leqno]{article}
\usepackage{amsmath}
\usepackage{amsfonts}
\usepackage{amssymb}
\usepackage{amsthm}
\usepackage{graphicx}
\usepackage{subfigure}
\usepackage{geometry}
\usepackage[utf8]{inputenc}
\usepackage{float}

\newtheorem{prop}{Proposition}

\newtheorem{rem}{Remark}
\newtheorem{example}{Example}

\newcommand{\ds}{\displaystyle}

\def\pa{\partial}

\title{A finite volume scheme for nonlinear degenerate parabolic equations}
\author{Marianne BESSEMOULIN-CHATARD and Francis FILBET}
\date{\today}

\begin{document}

\maketitle

\begin{abstract}
We propose a second order finite volume scheme for nonlinear degenerate parabolic equations which admit an entropy functional. For some of these models (porous media equation, drift-diffusion system for semiconductors, ...) it has been proved that the transient solution converges to a steady-state when time goes to infinity. The present scheme preserves steady-states and provides a satisfying long-time behavior. Moreover, it remains valid and second-order accurate in space even in the degenerate case. After describing the numerical scheme, we present several numerical results which confirm the high-order accuracy in various regime degenerate and non degenerate cases and underline the efficiency to preserve the large-time asymptotic.
\end{abstract}


\tableofcontents

\section{Introduction}

\paragraph*{}In this paper we propose  a second order accurate finite volume scheme for solving the following nonlinear, possibly degenerate parabolic equation: for $u: \mathbb{R}^+\times \Omega \mapsto \mathbb{R}^+$ solution to
\begin{equation}
\left\{
\begin{array}{l}
\pa_{t}u=\text{div}\, \left(f(u) \nabla V(x)+\nabla r(u)\right), \quad x \in \Omega, \quad t>0, 
\\
\,
\\
u(t=0,x)=u_{0}(x),
\end{array}\right. 
\label{eqgene}
\end{equation}
where  $\Omega \subset \mathbb{R}^{d}$ is an open bounded domain or the whole space $ \mathbb{R}^{d}$, $u \geq 0$ is a time-dependent density, $f$ is a given function and $r \in \mathcal{C}^{1}(\mathbb{R}_{+})$ is such that $r'(u) \geq 0$ and $r'(u)$ can vanish for certain values of $u$. Moreover, we assume that $r$ and $f$ are such that there exists a function $h$ such that $r'(s)=h'(s)\,f(s)$, and that $f(u) \geq 0$. This assumption means that the problem we consider has a structure corresponding to an energy or entropy, or more generally a Lyapunov functional. Or aim is to design a scheme which preserves this physical property. Indeed, the equation \eqref{eqgene} can be now written as
\begin{equation}
\pa_{t}u=\text{div}\left(f(u)\,\nabla \left( V(x)+h(u)\right)\right),
\label{eqgene2}
\end{equation}
and this equation admits an entropy functional, obtained by multiplying \eqref{eqgene2} by $\left(V+h(u)\right)$ and integrating over $\Omega$; it yields
\begin{equation*}
\frac{dE(t)}{dt}=-\mathcal{I}(t) \leq 0,
\end{equation*}
where the entropy $E$ is defined by 
\begin{equation}
E(t):= \int_{\Omega}u \left(V+h(u)\right)\,dx
\label{entropygene}
\end{equation} 
and the entropy dissipation $\mathcal{I}$ is given by
\begin{equation}
\mathcal{I}(t)=\int_{\Omega}f(u)\left|\nabla \left(V+h(u)\right)\right|^{2}\,dx.
\label{dissipationgene}
\end{equation}

\paragraph*{}A large variety of numerical methods have been proposed for the discretization of nonlinear degenerate parabolic equations: piecewise linear finite elements \cite{Barrett1997,Ebmeyer1998,Jager1991,Nochetto2000,Nochetto1988}, cell-centered finite volume schemes \cite{Eymard2000,Eymard2002}, vertex-centered finite volume schemes \cite{Ohlberger2001}, finite difference methods \cite{Karlsen2002}, mixed finite element methods \cite{Arbogast1996}, local discontinuous Galerkin finite element methods \cite{Zhang2009}, combined finite volume-finite element approach \cite{Eymard2006a}. Schemes based on discrete BGK models have been proposed in \cite{Aregba-Driollet2004}, as well as characteristics-based methods considered in \cite{Chen2003,Kavcur2002}. Other approaches are either based on a suitable splitting technique  \cite{Evje1999}, or based on the maximum principle and on perturbation and regularization \cite{Pop2002}. Also high order schemes have been developed in \cite{Cavalli2007,Liu2011,Kurganov2000}, which is a crucial step getting  an accurate approximation of the transient solution. 
\\
In this paper our aim is to construct a second-order finite volume scheme preserving  steady-states in order to obtain a satisfying long-time behavior for numerical solutions. Indeed, it has been observed in \cite{Bessemoulin-Chatard} that numerical schemes based on the preservation of steady states for degenerate parabolic problems offer a very accurate behavior of the approximate solution as time goes to infinity. To our knowledge, only few papers investigate this large-time asymptotic of numerical solutions. L. Gosse and G. Toscani proposed in \cite{Gosse2006} a scheme based on a formulation using the pseudo-inverse of the density's repartition function for porous media equation and fast-diffusion equation, and analysed the long-time behavior of approximate solutions. C. Chainais-Hillairet and F. Filbet studied in \cite{Chainais-Hillairet2007} a finite volume discretization for nonlinear drift-diffusion system and proved that the numerical solution converges to a steady-state when time goes to infinity. In \cite{Burger2010}, M. Burger, J. A. Carrillo and M. T. Wolfram proposed a mixed finite element method for nonlinear diffusion equations and proved convergence towards the steady-state in case of a nonlinear Fokker-Planck equation with uniformly convex potential. Here we propose a general way for designing a high-order scheme for nonlinear degenerate parabolic equations \eqref{eqgene} admitting an entropy functional. This scheme preserves steady-states and entropy decay like those proposed in \cite{Bessemoulin-Chatard,Chainais-Hillairet2007,Gosse2006}. Moreover, it appears that a loss of accuracy can happen when the problem degenerates, causing a deterioration of the long-time behavior of the approximate solution. Our new scheme tackles this issue since it remains second-order accurate in space both in degenerate and non-degenerate regimes.

\paragraph*{} Before describing our numerical scheme, let us emphasize that for some models described by equation (\ref{eqgene}), the large-time asymptotic has been studied using entropy/entropy-dissipation arguments, which will be the starting point of our approach. On the one hand equation (\ref{eqgene}) with linear convection, namely $f(u)=u$, has been analysed by J.A. Carrillo, A. Jüngel, P. A. Markowich, G. Toscani and A. Unterreiter in \cite{Carrillo2001}. On the other hand for equation (\ref{eqgene}) with nonlinear convection and linear diffusion a particular case has been studied in \cite{Carrillo2008,Carrillo2009,Toscani2012} by J. A. Carrillo, Ph. Laurençot, J. Rosado, F. Salvarani and G. Toscani. We will now remind some of the useful results contained in these papers.


\paragraph{Case of a linear convection.} The paper \cite{Carrillo2001} focuses on the long time asymptotic with exponential decay rate for 
\begin{equation}
\pa_{t}u=\text{div}\, \left(u \nabla V(x)+\nabla r(u)\right), \quad x \in \Omega, \quad t>0, 
\label{eqgenediff}
\end{equation}
with initial condition $u(t=0,x)=u_{0}(x) \geq 0$, $u_{0} \in L^{1}(\Omega)$ and 
\begin{equation*}
\int_{\Omega}u_{0}(x)\, dx =:M.
\end{equation*} 
Equation (\ref{eqgenediff}) is supplemented either by a decay condition when $|x| \rightarrow \infty$ if $\Omega= \mathbb{R}^{d}$ or by a zero out-flux condition on $\pa \Omega$ if $\Omega$ is bounded. In the following, we assume that $r:\mathbb{R}_{+} \rightarrow \mathbb{R}$ belongs to $ \mathcal{C}^{2}(\mathbb{R}_{+})$, is increasing and verifies $r(0)=0$. We define 
\begin{equation}
h(s):=\int_{1}^{s}\frac{r'(\tau)}{\tau}\, d\tau, \quad s \in (0,\infty),
\label{defh}
\end{equation}
and assume that $h \in L^{1}_{loc}\left([0,\infty)\right)$. Then
\begin{equation}
H(s):=\int_{0}^{s}h(\tau)\, d\tau, \quad s \in [0, \infty),
\label{defH}
\end{equation}
is well-defined, and $H'(s)=h(s)$ for all $s \geq 0$.\\
To analyze the large-time behavior to (\ref{eqgenediff}), stationary solutions $u^{eq}$ of (\ref{eqgenediff}) in $\Omega$ are first studied:
\begin{equation*}
u^{eq}\nabla V(x)+\nabla r(u^{eq})=0, \quad \int_{\Omega}u^{eq}(x)\, dx=M.
\end{equation*}
By using the definition (\ref{defh}) of $h$, this can be written as
\begin{equation*}
u^{eq} \left( \nabla V(x)+\nabla h(u^{eq}) \right)=0, \quad \int_{\Omega}u^{eq}(x) \, dx=M,
\end{equation*}
and if $u^{eq}>0$ in $\Omega$, then one obtains
\begin{equation*}
V(x)+h \left(u^{eq}(x)\right)=C \quad \forall x \in \Omega,
\end{equation*}
for some $C \in \mathbb{R}$. By considering the entropy functional
\begin{equation*}
E(u):=\int_{\Omega}\left( V(x)\,u(x)\,+\,H(u(x))\right)\, dx,
\end{equation*}
a function $u^{eq,M} \in L^{1}(\Omega)$ is an equilibrium solution of (\ref{eqgenediff}) if and only if it is a minimizer of $E$ in 
$$
\mathcal{C}\,=\,\left\{ u \in L^{1}(\Omega), \,\, \int_{\Omega}u(x) \, dx=M \right\}.
$$ 
Under some regularity assumptions on $V$, existence and uniqueness of an equilibrium solution is proved.  Therefore, the long time behavior is investigated and the exponential decay of the relative entropy
\begin{equation}
\mathcal{E}\left(t\right):=E\left(u(t)\right)-E(u^{eq,M})
\label{defRE}
\end{equation}
is shown, using the exponential decay of the entropy dissipation
\begin{equation*}
\mathcal{I}\left(t\right)\,:=\,-\frac{d\mathcal{E}(t)}{dt}\,=\,\int_{\Omega}u(t,x)\left| \nabla \left(V(x)\,+\,h(u(t,x))\right)\right|^{2}\,dx.
\end{equation*}
Finally using a generalized Csiszar-Kullback inequality, it is proved that the solution $u(t,x)$ of (\ref{eqgenediff}) with $r(s)=\log(s)$ or $r(s)=s^{m}$, $m \geq 0$, converges to the equilibrium $u^{eq,M}(x)$ as $t \rightarrow \infty$ at an exponential rate.

Equation (\ref{eqgenediff}) includes many well-known equations governing physical phenomena as porous media or drift-diffusion models for semiconductors.
 
\begin{example}[the porous media equation] In the case $V(x)=|x|^2/2$ and $r(u)=u^{m}$, with $m>1$, equation (\ref{eqgenediff}) is the porous media equation, which describes the flow of a gas through a porous interface. J. A. Carrillo and G. Toscani have proved in \cite{Carrillo2000} that the unique stationary solution of the porous media equation is given by Barenblatt-Pattle type formula
\begin{equation}
u^{eq}(x)=\left(C_{1}-\frac{m -1}{2m}|x|^{2}\right)_{+}^{1/(m -1)},
\label{barenblatt}
\end{equation}
where $C_{1}$ is a constant such that $u^{eq}$ has the same mass as the initial data $u_{0}$. Moreover, the convergence of the solution $u(t,x)$ of the porous media equation to the Barenblatt-Pattle solution $u^{eq}(x)$ as $t \rightarrow \infty$ has been proved in \cite{Carrillo2000}, using the entropy method. 
\end{example}
\begin{example}[the drift-diffusion model for semiconductors] The drift-diffusion model can also be interpreted in the formalism of (\ref{eqgenediff}). It is written as
\begin{equation}
\left\{\begin{array}{lcl} \pa_{t}N-\nabla \cdot (\nabla r(N)-N\nabla V)=0 , 
\\
\,
\\
\pa_{t}P-\nabla \cdot(\nabla r(P)+P\nabla V)=0,  
\\
\,
\\
\Delta V=N-P-C,\end{array}\right.
\label{DD}
\end{equation}
where the unknowns are $N$ the electron density, $P$ the hole density and $V$ the electrostatic potential, and $C$ is the prescribed doping profile. The two continuity equations on the densities $N$ and $P$ correspond to (\ref{eqgenediff}) with $r(s)=s^{\gamma}$ the pressure function. These equations are supplemented with initial conditions $N_{0}(x)$ and $P_{0}(x)$ and physically motivated boundary conditions: Dirichlet boundary conditions $ \overline{N}$, $ \overline{P}$ and $ \overline{V}$ on ohmic contacts $\Gamma^{D}$ and homogeneous Neumann boundary conditions on insulating boundary segments $\Gamma^{N}$.\\
The stationary drift-diffusion system admits a solution $(N^{eq},P^{eq},V^{eq})$ (see \cite{Markowich1993}), which is unique if in addition:
\begin{equation}
h(N^{eq})-V^{eq} \left\{ \begin{array}{lll} = \alpha_{N} & \text{ if } & N^{eq}>0 \\ \geq \alpha_{N} & \text{ if } & N^{eq}=0\end{array}\right., \quad h(P^{eq})+V^{eq} \left\{ \begin{array}{lll} = \alpha_{P} & \text{ if } & P^{eq}>0 \\ \geq \alpha_{P} & \text{ if } & P^{eq}=0\end{array}\right.,
\label{compatibility1}
\end{equation}
holds, and if the Dirichlet boundary conditions satisfy (\ref{compatibility1}) and the compatibility condition (if $ \overline{N}\,\overline{P}>0$)
\begin{equation}
h(\overline{N})+h(\overline{P})=\alpha_{N}+\alpha_{P}.
\label{compatibility2} 
\end{equation}
In this case the thermal equilibrium $(N^{eq},P^{eq},V^{eq})$ is defined by
\begin{equation}
\left\{ \begin{array}{rcl} \Delta V^{eq}=g\left(\alpha_{N}+V^{eq}\right)-g\left(\alpha_{P}-V^{eq}\right)-C & &  \text{on } \Omega,
\\
\,
\\ N^{eq}=g\left(\alpha_{N}+V^{eq}\right), \,\ P^{eq}=g\left(\alpha_{P}-V^{eq}\right) & & \text{on } \Omega,\end{array}\right.
\label{eqthermiqueDD}
\end{equation}
where $g$ is the generalized inverse of $h$, namely
\begin{equation*}
g(s)=\left\{\begin{array}{lcl} h^{-1}(s) & \text{ if } & h(0_{+})<s<\infty, \\ 0 & \text{ if } & s \leq h(0_{+}). \end{array}\right.
\end{equation*}
In the linear case $r(u)=u$, it has been proved by H. Gajewski and K. Gärtner in \cite{Gajewski1996} that the solution to the transient system (\ref{DD}) converges to the thermal equilibrium state as $t \rightarrow \infty$ if the boundary conditions are in thermal equilibrium. A. Jüngel extends this result to a degenerate model with nonlinear diffusion in \cite{Juengel1995}.    In both cases the key-point of the proof is an energy estimate with the control of the energy dissipation.
\end{example}


\paragraph{Case of a nonlinear convection.}
In \cite{Carrillo2008,Carrillo2009,Toscani2012}, a nonlinear Fokker-Planck type equation modelling the relaxation of fermion and boson gases is studied. This equation corresponds to (\ref{eqgene}) with linear diffusion and nonlinear convection:
\begin{equation}
\pa_{t}u=\text{div} \, \left(xu(1+ku)+\nabla u\right), \quad x \in \mathbb{R}^{d}, \quad t>0,
\label{bosonfermion}
\end{equation}
with $k=1$ in the boson case and $k=-1$ in the fermion case. The long-time asymptotic of this model has been studied in 1D for both cases \cite{Carrillo2008}, in any dimension for fermions \cite{Carrillo2009} and in 3D for bosons \cite{Toscani2012}. The stationary solution of (\ref{bosonfermion}) is given by the Fermi-Dirac ($k=-1)$ and Bose-Einstein ($k=1$) distributions:
\begin{equation}
u^{eq}(x)=\frac{1}{\beta e^{\frac{|x|^{2}}{2}}-k},
\label{eqbosonfermion}
\end{equation}
where $ \beta \geq 0$ is such that $u^{eq}$ has the same mass as the initial data $u_{0}$. The entropy functional is given by
\begin{equation*}
E(u):=\int_{\mathbb{R}^{d}}\left(\frac{|x|^{2}}{2}u+u\log (u)-k(1+ku)\log (1+ku)\right)\, dx,
\end{equation*}
and the entropy dissipation is defined by
\begin{equation*}
\mathcal{I}(t)\,:=-\frac{d\mathcal{E}(t)}{dt}\,=\,\int_{\mathbb{R}^{d}}u(1+ku)\left| \nabla \left(\frac{|x|^{2}}{2}+\log\left(\frac{u}{1+ku}\right)\right)\right|^{2}\, dx.
\end{equation*}
Then decay rates towards equilibrium are given in \cite{Carrillo2008,Carrillo2009} for fermion case in any dimension and for 1D boson case by relating the entropy and its dissipation. As in the case of a linear diffusion, the key-point of the proof is an entropy estimate with the control of its dissipation. \\
Concerning 3D boson case, it is proved in \cite{Toscani2012} that for sufficiently large initial mass, the solution blows up in finite time.\\
Let us also mention that a more general class of Fokker-Planck type equations for bosons with linear diffusion and super-linear drift is studied in \cite{BenAbdallah2011}:
\begin{equation}\label{eqbosons.gene}
\pa_{t}u=\text{div}(xu(1+u^{N})+\nabla u),
\end{equation}
where $N>0$ is a given constant. For $N>2$, there is a phenomenon of critical mass in dimension 1. It is proved by minimizing an entropy functional that starting from an initial distribution with a super-critical mass, the solution develops a singular part localized in the origin.


\paragraph{} As explained above, it has been proved by entropy/entropy dissipation techniques  that the solution to (\ref{eqgene}) converges to a steady-state as time goes to infinity often with an exponential time decay rate. Our aim is to propose a numerical scheme considering these problems and for which we can obtain a discrete entropy estimate as in the continuous case. In \cite{Arnold2003,Carrillo2007,Burger2010} temporal semi-discretizations have been proposed and semi-discrete entropy estimates have been proved. However, when the problem is spatially discretized  a  saturation of the entropy and its dissipation may appear, due to the spatial discretization error. This emphasizes the importance of considering  spatial discretization techniques which preserve the steady-states and the entropy dissipation. This point of view has been already adopted in \cite{Bessemoulin-Chatard,Chainais-Hillairet2007} but both schemes do not provide really satisfying results when the equation degenerates. Indeed both schemes degenerate in the upwind flux if the diffusion vanishes and then are only first order accurate in space. Thus we propose in this paper a finite volume scheme for nonlinear parabolic equations, possibly degenerate, possessing an entropy functional. We focus on the spatial discretization, with a twofold objective. On the one hand we require preserving steady-states in order to obtain a satisfying long-time behavior of the approximate solution. On the other hand the scheme proposed remains valid and second order accurate in space even in the degenerate case. The main idea of our new scheme is to discretize together the convective and diffusive parts of the equation (\ref{eqgene}) to obtain a flux which preserves equilibrium and to use a slope-limiter method to get second-order accuracy even in the degenerate case.

\paragraph{} The plan of the paper is as follows. In Section 2, we construct the finite volume scheme. We first focus on the case of a linear diffusion (\ref{eqgenediff}). Then we extend this construction to the general case (\ref{eqgene}). In Section 3 we give some basic properties of the scheme and a semidiscrete entropy estimate for the case of a linear diffusion (\ref{eqgenediff}). We end in Section 4 by presenting some numerical results. We first verify experimentally the second order accuracy in space of our scheme, even in the degenerate case. Then we focus on the long-time behavior. The scheme is applied to the physical models introduced above and the numerical results confirm its efficiency to preserve the large-time asymptotics. Finally we propose a test case with both nonlinear convection and diffusion.


\section{Presentation of the numerical scheme}


In this section we present our new finite volume scheme for (\ref{eqgene}). For simplicity purposes, we consider the problem in one space dimension. It will be straightforward to generalize this construction for Cartesian meshes in multidimensional case.\\ 
In a one-dimensional setting, $\Omega=(a,b)$ is an interval of $ \mathbb{R}$. We consider a mesh for the domain $(a,b)$, which is not necessarily uniform $ \textit{i.e.}$ a family of $N_{x}$ control volumes $\left(K_{i}\right)_{i=1,...,N_{x}}$ such that $K_{i}=\left]x_{i-\frac{1}{2}},x_{i+\frac{1}{2}}\right[$ with $\ds{ x_{i}=(x_{i-\frac{1}{2}}+x_{i+\frac{1}{2}})/2 }$ and 
\begin{equation*}
a=x_{\frac{1}{2}}<x_{1}<x_{\frac{3}{2}}<...<x_{i-\frac{1}{2}}<x_{i}<x_{i+\frac{1}{2}}<...<x_{N_{x}}<x_{N_{x}+\frac{1}{2}}=b.
\end{equation*}
Let us set
\begin{eqnarray*}
\Delta x_{i}=x_{i+\frac{1}{2}}-x_{i-\frac{1}{2}}, \quad \text{ for } 1 \leq i \leq N_{x},\\
\Delta x_{i+\frac{1}{2}} = x_{i+1}-x_{i}, \quad \text{ for } 1 \leq i \leq N_{x}-1.
\end{eqnarray*}
Let $ \Delta t$ be the time step. We set $t^{n}=n \Delta t$. A time discretization of $(0,T)$ is then given by the integer value $N_{T}=E(T/\Delta t)$ and by the increasing sequence of $(t^{n})_{0\leq n \leq N_{T}}$.\\
First of all, the initial condition is discretized on each cell $K_{i}$ by:
\begin{equation*}
U_{i}^{0}=\frac{1}{\Delta x_{i}}\int_{K_{i}}u_{0}(x)\, dx, \quad i=1,...L.
\end{equation*}
The finite volume scheme is obtained by integrating the equation (\ref{eqgene}) over each control volume $K_{i}$ and over each time step. Concerning the time discretization, we can choose any explicit method (forward Euler, Runge-Kutta,...). Since in this paper we are interested in the spatial discretization, we will only consider a forward Euler method afterwards. Let us now focus on the spatial discretization.\\
We denote by $U_{i}(t)$ an approximation of the mean value of $u$ over the cell $K_{i}$ at time $t$. By integrating the equation (\ref{eqgene}) on $K_{i}$, we obtain the semi-discrete numerical scheme:
\begin{equation}
\Delta x_{i} \frac{d}{dt}U_{i} + \mathcal{F}_{i+\frac{1}{2}}-\mathcal{F}_{i- \frac{1}{2}}\,=\,0,
\label{semidiscrete}
\end{equation}
where $\mathcal{F}_{i+\frac{1}{2}}$ is an approximation of the flux $-\left[f(u)\pa_{x}V\,+\,\pa_{x}r(u)\right]$ at the interface $x_{i+\frac{1}{2}}$ which remains to be defined. 


\paragraph{Case of a linear convection ($f(u)=u$).} To explain our approach we first define the numerical flux for equation (\ref{eqgenediff}). The main idea is to discretize together the convective and the diffusive parts. To this end, we write $\left[u\pa_{x}V+\pa_{x}r(u)\right]$ as $u\left[\pa_{x}\left(V+h(u)\right)\right]$, where $h$ is defined by (\ref{defh}). Then we will consider $-\pa_{x}\left(V+h(u)\right)$ as a velocity and denote by $A_{i+\frac{1}{2}}$ an approximation of this velocity at the interface $x_{i+\frac{1}{2}}$:
\begin{equation*}
A_{i+\frac{1}{2}}=-dV_{i+\frac{1}{2}}-dh(U)_{i+\frac{1}{2}},
\end{equation*}
where $dV_{i+\frac{1}{2}}$ and $dh(U)_{i+\frac{1}{2}}$ are centered approximations of $\pa_{x}V$ and $\pa_{x}h(u)$ respectively, namely
\begin{equation*}
dV_{i+\frac{1}{2}}=\frac{V(x_{i+1})-V(x_{i})}{\Delta x_{i+\frac{1}{2}}}, \quad dh(U)_{i+\frac{1}{2}}=\frac{h(U_{i+1})-h(U_{i})}{\Delta x_{i+\frac{1}{2}}}.
\end{equation*}
Now we apply the standard upwind method and then define our new numerical flux, called fully upwind flux, as
\begin{equation}
\mathcal{F}_{i+\frac{1}{2}}=F(U_{i},U_{i+1})=A_{i+\frac{1}{2}}^{+}U_{i}-A_{i+\frac{1}{2}}^{-}U_{i+1},
\label{defF1}
\end{equation}
where $x^{+}=\max(0,x)$ and $x^{-}=\max(0,-x)$. This method is only first-order accurate. To obtain second-order accuracy, we replace in (\ref{defF1}) $U_{i}$ and $U_{i+1}$ by $U_{i+\frac{1}{2},-}$ and $U_{i+\frac{1}{2},+}$ respectively, which are reconstructions of $u$ at the interface defined by: 
\begin{equation}
\left\{\begin{array}{l}
U_{i+\frac{1}{2},-} \,=\, U_{i}+\frac{1}{2}\phi \left(\theta_{i}\right)\left(U_{i+1}-U_{i}\right), 
\\
\,
\\
U_{i+\frac{1}{2},+} \;=\, U_{i+1}-\frac{1}{2}\phi \left(\theta_{i+1}\right)\left(U_{i+2}-U_{i+1}\right),
\end{array}\right.
\label{defudem}
\end{equation}
with
\begin{equation*}
\theta_{i}=\frac{U_{i}-U_{i-1}}{U_{i+1}-U_{i}}
\end{equation*}
and $\phi$ is a slope-limiter function (setting $\phi=0$ gives the classical upwind flux). From now on we will consider the second-order fully upwind scheme defined with the Van Leer limiter:
\begin{equation*}
\phi(\theta)=\frac{\theta+|\theta|}{1+|\theta|}.
\end{equation*}


\paragraph{General case.} We now consider the general case where both diffusion and convection are nonlinear in (\ref{eqgene}). We assume that $f(u) \geq 0$ and that we can define $ h(u)$ such that $ h'(u)f(u)=r'(u)$. Then the equation \eqref{eqgene} admits an entropy, as explained in the introduction. Following the same idea as above, we use the following expression of the flux
\begin{equation}
f(u)\pa_{x}V+\pa_{x}r(u)\,\,=\,\,\pa_{x}\left(V+h(u)\right)\,f(u),
\label{fluxgene}
\end{equation} 
and define the numerical flux as a local Lax-Friedrichs:
\begin{equation}
\mathcal{F}_{i+\frac{1}{2}}=\frac{{A}_{i+\frac{1}{2}}}{2}\left(f(U_{i})+f(U_{i+1})\right)-\frac{\left|{A}_{i+\frac{1}{2}}\right|\alpha_{i+\frac{1}{2}}}{2}\left(U_{i+1}-U_{i}\right),
\label{defFgene}
\end{equation}
where 
\begin{equation*}
A_{i+\frac{1}{2}}=-dV_{i+\frac{1}{2}}-dh(U)_{i+\frac{1}{2}},
\end{equation*}
and
\begin{equation*}
\alpha_{i+\frac{1}{2}}=\max \left( \left|f'(u)\right|\right) \text{ over all }u \text{ between } U_{i} \text{ and } U_{i+1}.
\end{equation*}
As above, we replace $U_{i}$ and $U_{i+1}$ in (\ref{defFgene}) by reconstructions $U_{i+\frac{1}{2},-}$ and $U_{i+\frac{1}{2},+}$ defined by (\ref{defudem}) to obtain a second-order scheme.\\

We can now summarize our new numerical flux by:
\begin{equation}
\left\{\begin{array}{lll}
 \ds{\mathcal{F}_{i+\frac{1}{2}}=\frac{{A}_{i+\frac{1}{2}}}{2}\left(f(U_{i+\frac{1}{2},-})+f(U_{i+\frac{1}{2},+})\right)-\frac{\left|{A}_{i+\frac{1}{2}}\right|\alpha_{i+\frac{1}{2}}}{2}\left(U_{i+\frac{1}{2},+}-U_{i+\frac{1}{2},-}\right),}& &
\\
 \ds{{A}_{i+\frac{1}{2}}=-dV_{i+\frac{1}{2}}-dh(U)_{i+\frac{1}{2}}, \vphantom{\frac{A_{i+\frac{1}{2}}}{2}}}& &\\
\ds{\alpha_{i+\frac{1}{2}}=\max \left( \left|f'(u)\right|\right) \text{ over all }u \text{ between } U_{i} \text{ and } U_{i+1}, \vphantom{\frac{A_{i+\frac{1}{2}}}{2}}}& & \\
\ds{U_{i+\frac{1}{2},-} = U_{i}+\frac{1}{2}\phi \left(\theta_{i}\right)\left(U_{i+1}-U_{i}\right),\vphantom{\frac{A_{i+\frac{1}{2}}}{2}}}& &  \\
\ds{U_{i+\frac{1}{2},+} = U_{i+1}-\frac{1}{2}\phi \left(\theta_{i+1}\right)\left(U_{i+2}-U_{i+1}\right), \vphantom{\frac{A_{i+\frac{1}{2}}}{2}}} & &\end{array}\right.
\label{bigdef}
\end{equation}
where either a first-order scheme 
\begin{equation}
\phi(\theta)=0, 
\label{order1}
\end{equation}
or a second order scheme
\begin{equation}
\phi(\theta)=\frac{\theta+|\theta|}{1+|\theta|}.
\label{order2}
\end{equation}

\begin{rem}[Generalization to multidimensional case]
It is straightforward to define the scheme for Cartesian meshes in multidimensional case: the 1D formula can be used as it is in any of the Cartesian directions. However, the construction of the scheme on unstructured meshes is more complicated. More precisely, it is easy to define the first order scheme on such grids, but the difficulty is to obtain high-order accuracy. As in the one dimensional case, the idea is to replace the first-order flux $F(U_{i},U_{j})$, where $U_{i}$, $U_{j}$ are the constant values on each side of an edge $\Gamma_{ij}=K_{i} \cap K_{j}$, by $F(U_{ij},U_{ji})$, where $U_{ij}$, $U_{ji}$ are second-order approximations of the solution on each side of the edge $ \Gamma_{ij}$. More precisely, we need to obtain piecewise linear functions on each triangle instead of piecewise constant functions. For more details concerning these questions, see for example \cite{Durlofsky1992,Godlewski1996} and the references therein.
\end{rem}


\section{Properties of the scheme}


In this section, we present some important properties of the scheme. We would like to emphasize here the preservation of the equilibrium and the entropy estimate, which are two crucial properties to study the scheme. Concerning a more advanced analysis of the scheme, we can apply the same techniques as in \cite{Bessemoulin-Chatard}, but this is not our purpose here.

\subsection{The semi-discrete scheme}

In this part, we study the semi-discrete scheme (\ref{semidiscrete})-(\ref{bigdef})-(\ref{order2}) and consider the equation (\ref{eqgenediff}) on a bounded domain with homogeneous Neumann boundary conditions. We assume that $r \in \mathcal{C}^{1}(\mathbb{R}_{+})$ is strictly increasing and $h$ is defined by (\ref{defh}). Then we consider a primitive $H$ of $h$, which is strictly convex since $r$ is strictly increasing.\\
We denote by $\left(U_{i}^{eq}\right)_{i=1,...,N_{x}}$ an approximation of the equilibrium solution $u^{eq}$. This approximation verifies 
\begin{equation}
dh\left(U^{eq}\right)_{i+\frac{1}{2}}+dV_{i+\frac{1}{2}}=0 \quad \forall i=0,...,N_{x},
\label{eqapp}
\end{equation}
and
\begin{equation*}
\sum_{i=1}^{N_{x}} \Delta x_{i}U_{i}^{eq}=\sum_{i=1}^{N_{x}}\Delta x_{i}U_{i}^{0}=:\overline{M}.
\end{equation*}
A semi discrete version of the relative entropy $\mathcal{E}$ defined by (\ref{defRE}) is given by
\begin{equation}
\mathcal{E}_\Delta(t) \,:=\,\sum_{i=1}^{N_{x}}\Delta x_{i}\big(H\left(U_{i}(t)\right)-H\left(U_{i}^{eq}\right)-h\left(U_{i}^{eq}\right)\left(U_{i}(t)-U_{i}^{eq}\right)\big).
\label{Esd}
\end{equation}
We also introduce the semi discrete version of the entropy dissipation
\begin{equation*}
\mathcal{I}_\Delta(t)\,:=\,\sum_{i=0}^{N_{x}}\Delta x_{i+\frac{1}{2}}\left|A_{i+\frac{1}{2}}\right|^{2}\min\left(U_{i+\frac{1}{2},-}(t),U_{i+\frac{1}{2},+}(t)\right).
\end{equation*}

\begin{prop}
Assume that the initial data $U_i(0)$ is nonnegative. Then, the finite volume scheme (\ref{semidiscrete})-(\ref{bigdef})-(\ref{order2}) for equation (\ref{eqgenediff}) satisfies 
\begin{itemize}
\item[(i)] the preservation of the nonnegativity of $U_{i}(t)$,
\item[(ii)] the preservation of the equilibrium,
\item[(iii)] the entropy estimate: for $0<t_{1}\leq t_{2}<\infty$,
\begin{equation*}
0 \,\leq\, {\mathcal{E}}_{\Delta}(t_{2})\,+\,\int_{t_{1}}^{t_{2}}{\mathcal{I}}_{\Delta}(t) \,dt \,\leq \,{\mathcal{E}}_{\Delta}(t_{1}).
\end{equation*}
\end{itemize}
\end{prop}

\begin{proof}
To prove the preservation of nonnegativity, we need to check that 
\begin{equation}
F \left(U_{i+\frac{1}{2},-},U_{i+\frac{1}{2},+}\right)-F \left(U_{i-\frac{1}{2},-},U_{i-\frac{1}{2},+}\right) \leq 0
\label{nonnegativity}
\end{equation}
whenever $U_{i}=0$.\\
When $U_{i}=0$, we have $U_{i} \leq U_{i+1}$ and $U_{i} \leq U_{i-1}$, and then $ \theta_{i} \leq 0$, which gives $\phi(\theta_{i})=0$ and finally
\begin{equation*}
U_{i+\frac{1}{2},-}=U_{i-\frac{1}{2},+}=U_{i}=0.
\end{equation*}
Then we get
\begin{equation*}
F \left(U_{i+\frac{1}{2},-},U_{i+\frac{1}{2},+}\right)-F \left(U_{i-\frac{1}{2},-},U_{i-\frac{1}{2},+}\right) = -A_{i+\frac{1}{2}}^{-}U_{i+\frac{1}{2},+}-A_{i-\frac{1}{2}}^{+}U_{i-\frac{1}{2},-}.
\end{equation*}
Moreover, $U_{i-\frac{1}{2},-}$  is given by  
$$
U_{i-\frac{1}{2},-} \,=\,\left(1-\frac{\phi(\theta_{i-1})}{2}\right)U_{i-1}, 
$$
which is nonnegative since $\phi(\theta) \leq 2$ for all $\theta$. 

On the other hand, we deal with $U_{i+\frac{1}{2},+}$, and get that either $ \theta_{i+1} \leq 0$, then $U_{i+\frac{1}{2},+}=U_{i+1}\geq 0$, or we have $ \theta_{i+1}>0$, that is $U_{i+2} \geq U_{i+1}$ and since $\phi(\theta) \leq 2 \theta$ for all $\theta \geq 0$, we get
\begin{equation*}
U_{i+\frac{1}{2},+}\geq U_{i+1}-\theta_{i+1} \left(U_{i+2}-U_{i+1}\right)=U_{i+1}-\left(U_{i+1}-U_{i}\right)=0.
\end{equation*}
We conclude that (\ref{nonnegativity}) always holds when $U_{i}=0$, which gives $(i)$.\\
The part $(ii)$ is clear by construction: at the equilibrium, we have $dh(U)_{i+\frac{1}{2}}+dV_{i+\frac{1}{2}}=0$, which is exactly $A_{i+\frac{1}{2}}=0$ and then $ \mathcal{F}_{i+\frac{1}{2}}=0$.\\
By definition (\ref{Esd}) of $\mathcal{E}_{\Delta}(t)$ and since $H'(s)=h(s)$ for all $s \geq 0$, we have
\begin{equation*}
\frac{d{\mathcal{E}_\Delta}}{dt}(t)=\sum_{i=1}^{N_{x}}\Delta x_{i}\left(h(\left(U_{i}(t)\right)-h(U_{i}^{eq})\right)\frac{dU_{i}}{dt}(t).
\end{equation*}
Using the numerical scheme (\ref{semidiscrete}), we get
\begin{equation*}
\frac{d {\mathcal{E}_\Delta}}{dt}(t)=-\sum_{i=1}^{N_{x}}\left(h(\left(U_{i}(t)\right)-h(U_{i}^{eq})\right)\left(\mathcal{F}_{i+\frac{1}{2}}-\mathcal{F}_{i-\frac{1}{2}}\right),
\end{equation*}
and then a discrete integration by parts yields (using the homogeneous Neumann boundary conditions)
\begin{equation*}
\frac{d{\mathcal{E}}_\Delta}{dt}(t)=\sum_{i=0}^{N_{x}}\Delta x_{i+\frac{1}{2}}\left(dh(U(t))_{i+\frac{1}{2}}-dh(U^{eq})_{i+\frac{1}{2}}\right)\mathcal{F}_{i+\frac{1}{2}}.
\end{equation*}
Since by (\ref{eqapp}) we have $dh\left(U^{eq}\right)_{i+\frac{1}{2}}=-dV_{i+\frac{1}{2}}$, we obtain
\begin{eqnarray*}
\frac{d\mathcal{E}_\Delta}{dt}(t) & = & -\sum_{i=0}^{N_{x}}\Delta x_{i+\frac{1}{2}}A_{i+\frac{1}{2}}\left(A_{i+\frac{1}{2}}^{+}U_{i+\frac{1}{2},-}(t)-A_{i+\frac{1}{2}}^{-}U_{i+\frac{1}{2},+}(t)\right)\\
& \leq & -\sum_{i=0}^{N_{x}}\Delta x_{i+\frac{1}{2}}\left|A_{i+\frac{1}{2}}\right|^{2}\min\left(U_{i+\frac{1}{2},-}(t),U_{i+\frac{1}{2},+}(t)\right).
\end{eqnarray*}
Finally we get $(iii)$ by integrating between $t_{1}$ and $t_{2}$.
\end{proof}


\subsection{The fully-discrete scheme}

In this part we consider the fully-discrete scheme obtained by using the forward Euler method. We denote by $U^{n}_{i}$ an approximation of the mean value of $u$ over the cell $K_{i}$ at time $t^{n}=n \Delta t$. The fully-discrete scheme is given by:
\begin{equation}
\text{m}(K_{i})\,\frac{U_{i}^{n+1}-U_{i}^{n}}{\Delta t}\,+\,\mathcal{F}_{i+\frac{1}{2}}^{n}-\mathcal{F}_{i-\frac{1}{2}}^{n}\,=\,0,
\label{fullydiscrete}
\end{equation}
where the numerical flux $ \mathcal{F}_{i+\frac{1}{2}}$ is defined by (\ref{bigdef})-(\ref{order2}).

\begin{prop}
For $n \geq 0$, assume that $U_{i}^{n} \geq 0$ for all $i=1,...,N_{x}$. Then under the CFL condition 
\begin{equation}
\Delta t\,\max_{i}\left|V(x_{i+1})-V(x_{i})-h(U_{i+1}^{n})+h(U_{i}^{n})\right| \,\,\leq \,\,\frac{1}{2}\,{\min_{i}}\,\Delta x_{i}^{2},
\label{CFL}
\end{equation} 
the fully-discrete first-order scheme (\ref{bigdef})-(\ref{order1}) and (\ref{fullydiscrete}) for equation (\ref{eqgenediff}) preserves the nonnegativity of $U_{i}$, which means that $U_{i}^{n+1} \geq 0$ for all $i=1,...,N_{x}$, and the steady-states solution.
\end{prop}

\begin{proof}
Using the definition (\ref{fullydiscrete})-(\ref{bigdef})-(\ref{order1}) of the fully-discrete first-order scheme, we get for all $i=1,...,N_{x}$
\begin{equation*}
U_{i}^{n+1}=\left(1-\frac{\Delta t}{\Delta x_{i}}\left(\left(A_{i+\frac{1}{2}}^{n}\right)^{+}+\left(A_{i-\frac{1}{2}}^{n}\right)^{-}\right)\right)U_{i}^{n}+\frac{\Delta t}{\Delta x_{i}}\left(A_{i+\frac{1}{2}}^{n}\right)^{-}U_{i+1}^{n}+\frac{\Delta t}{\Delta x_{i}}\left(A_{i-\frac{1}{2}}^{n}\right)^{+}U_{i-1}^{n}.
\end{equation*} 
Thus we deduce that $U_{i}^{n+1} \geq 0$ as soon as $\ds{\frac{\Delta t}{\Delta x_{i}}\left(\left(A_{i+\frac{1}{2}}^{n}\right)^{+}+\left(A_{i-\frac{1}{2}}^{n}\right)^{-}\right) \leq 1}$, which is necessarily the case from (\ref{CFL}), using the definition of $A_{i+\frac{1}{2}}^{n}$.
\end{proof}

\begin{rem}
This result is not surprising since the stability condition for an explicit discretization of a parabolic equation requires the time step to be limited by a power two of the space step. 
\end{rem}


\section{Numerical simulations}


In this section, we present several numerical results performed by using our new fully-upwind flux. In all the numerical experiments performed, since our purpose is to focus on the spatial discretization, we choose a forward Euler method for the time discretization. As explained above, this choice of an explicit time discretization implies that the time step has to be limited by the square of the space step. Since we want to study the spatial accuracy of our scheme, we voluntarily choose a small time step in the first part of this section. Furthermore, since the CFL condition \eqref{CFL} becomes far less restrictive when the problem degenerates or when the solution tends to the equilibrium, we can use an adaptative time step.  Nevertheless, a fully implicit scheme would be also suitable for the long time behavior of the numerical solution  since in that case the numerical solution satisfies an entropy inequality, which is not the case with the explicit discretization we choose.
\\

We first study the spatial order of convergence of the scheme for linear convection in both non degenerate and degenerate cases. Then we will apply it to the physical models presented in the introduction: the porous media equation, the drift-diffusion system for semiconductors and the nonlinear Fokker-Planck equation for bosons and fermions. The results underline the efficiency of the scheme to preserve long-time behavior of the solutions. Finally we apply the scheme to a fully nonlinear problem: the Buckley-Leverett equation.\\
Below we make comparison between the finite volume schemes (\ref{fullydiscrete}) defined with the following numerical fluxes:
\begin{itemize}
\item \textbf{The first-order fully upwind flux}, given by
\begin{equation}
\mathcal{F}_{i+\frac{1}{2}}=\frac{{A}_{i+\frac{1}{2}}}{2}\left(f(U_{i})+f(U_{i+1})\right)-\frac{\left|{A}_{i+\frac{1}{2}}\right|\alpha_{i+\frac{1}{2}}}{2}\left(U_{i+1}-U_{i}\right),
\label{fluxFU1}
\tag{\textbf{FU1}}
\end{equation}  
with $ \ds{{A}_{i+\frac{1}{2}}}$, $\ds{\alpha_{i+\frac{1}{2}}}$ defined in (\ref{bigdef}).
\item \textbf{The second-order fully upwind flux}, given by
\begin{equation}
\mathcal{F}_{i+\frac{1}{2}}=\frac{{A}_{i+\frac{1}{2}}}{2}\left(f(U_{i+\frac{1}{2},-})+f(U_{i+\frac{1}{2},+})\right)-\frac{\left|{A}_{i+\frac{1}{2}}\right|\alpha_{i+\frac{1}{2}}}{2}\left(U_{i+\frac{1}{2},+}-U_{i+\frac{1}{2},-}\right).
\label{fluxFU2}
\tag{\textbf{FU2}}
\end{equation}  
\item \textbf{The classical upwind flux,} introduced and studied in \cite{Eymard2000}. It is valid for linear convection and for both linear and nonlinear diffusion. The diffusion term is discretized classically by using a two-points flux and the convection term is discretized with the upwind flux. This flux has then been used for the drift-diffusion system for semiconductors \cite{Chainais-Hillairet2003,Chainais-Hillairet2003a,Chainais-Hillairet2004}. It is defined for equation (\ref{eqgenediff}) by
\begin{equation}
\mathcal{F}_{i+\frac{1}{2}} =\left(-dV_{i+\frac{1}{2}} \right)^{+}U_{i}-\left(-dV_{i+\frac{1}{2}} \right)^{-}U_{i+1}-\frac{r\left(U_{i+1}\right)-r\left(U_{i}\right)}{\Delta x_{i+\frac{1}{2}}}.
\label{fluxCU}
\tag{\textbf{CU}}
\end{equation}
\item \textbf{The Scharfetter-Gummel flux and its extension for nonlinear diffusion.} This scheme is widely used in the semiconductors framework in the case of a linear diffusion. It has been proposed in \cite{Il'in1969,Scharfetter1969} for the numerical approximation of the 1D drift-diffusion model. This scheme preserves equilibrium and is second-order accurate \cite{Lazarov1996}. The definition of the Scharfetter-Gummel flux has been extended to the case of a nonlinear diffusion in \cite{Bessemoulin-Chatard}. For equation (\ref{eqgenediff}) this flux is written 
\begin{equation}
\mathcal{F}_{i+\frac{1}{2}} =\frac{dr_{i+\frac{1}{2}}}{\Delta x_{i+\frac{1}{2}}}\left[ B\left(\frac{\Delta x_{i+\frac{1}{2}} dV_{i+\frac{1}{2}}}{dr_{i+\frac{1}{2}}}\right)U_{i}-B\left(-\frac{\Delta x_{i+\frac{1}{2}} dV_{i+\frac{1}{2}}}{dr_{i+\frac{1}{2}}}\right)U_{i+1}\right],
\label{fluxSG1}
\tag{\textbf{SGext}}
\end{equation}
where 
\begin{equation*}
\left\{\begin{array}{lll} \ds{B(x)=\frac{x}{e^{x}-1}} \text{ for } x \neq 0, \quad B(0)=1, & &\\  \ds{ dr_{i+\frac{1}{2}}=dr\left(U_{i},U_{i+1}\right) ,\vphantom{ \frac{x-\frac{1}{2}}{e^{x}}}} \end{array}\right.
\end{equation*}
with for $a$, $b \in \mathbb{R}_{+}$,
\begin{equation*}
dr(a,b)=\left\{ \begin{array}{cll} \ds{\frac{h(b)-h(a)}{\log(b)-\log(a)}} & &\text{ if } ab>0 \text{ and } a \neq b,\\ \, \\ \ds{r'\left(\frac{a+b}{2}\right)} & &\text{ elsewhere. } \end{array} \right.
\end{equation*} 
\end{itemize}


\subsection{Order of convergence}

In this part, we test the spatial accuracy of the scheme for linear convection ($f(s)=s$). We first consider a test case in 1D on $(0,T) \times (-1,1)$ with $ \pa_{x}V=-1$. The time step is taken equal to $\Delta t=10^{-8}$ to study the order of convergence with respect to the spatial step size. The boundary conditions are periodic. Since we don't know an exact solution of the problem, we compute relative errors. More precisely, an estimation of the relative error in $L^{1}$ norm at time $T$ is given by
\begin{equation*}
e_{2\Delta x}=\Vert u_{\Delta x}(T)-u_{2\Delta x}(T)\Vert_{L^{1}(\Omega)},
\end{equation*}
where $u_{\Delta x}$ represents the approximation computed from a mesh of size $\Delta x$. The numerical scheme is said to be $k$-th order if $e_{2\Delta x} \leq C \Delta x^{k}$, for all $0<\Delta x \ll 1$. \\
\paragraph*{Example 1 (Non degenerate case).}  We first take $r(s)=s^{2}$, thus $r'(0)=0$ and $r'(s)>0$ for all $s>0$. The initial data is
\begin{equation*}
u_{0}(x)= 0.5+0.5\sin(\pi x), \quad x \in (-1,1)
\end{equation*} 
and the final time $T=0.1$. In Figure \ref{ex1_1},  we represent the evolution of the approximate solution computed on a fine mesh made of 3200 cells, with the scheme (\textbf{FU2}). Since the solution becomes strictly positive for all $t>0$, this problem is not degenerate.
\begin{figure}[!ht]
\centering
\includegraphics[width=2.6in]{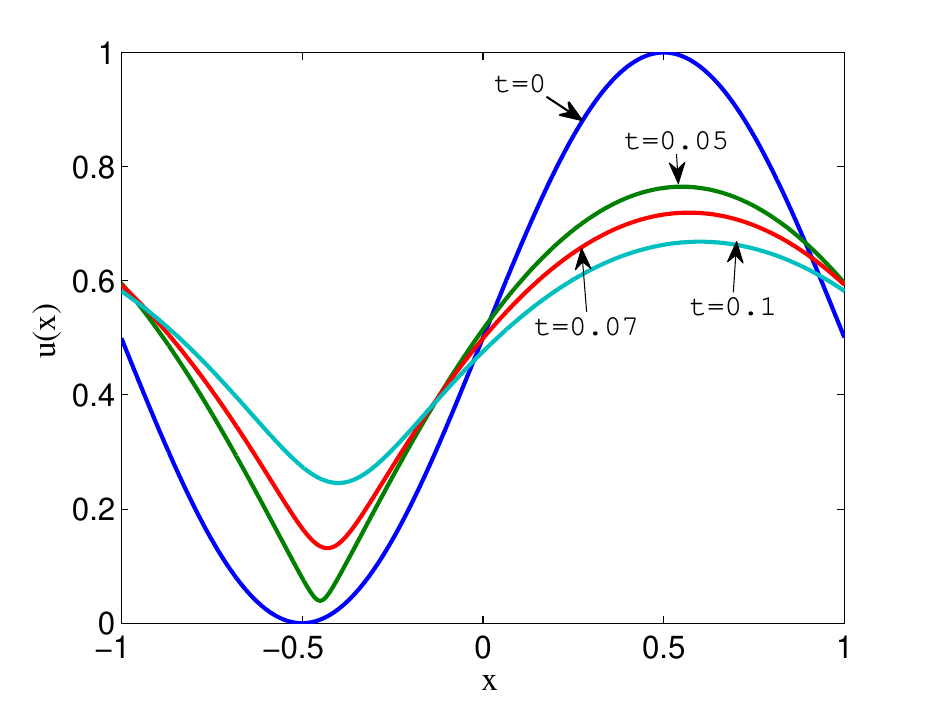}
\caption{Example 1 - Evolution of the approximate solution computed on a fine mesh.}
\label{ex1_1}
\end{figure}
In Table \ref{tablenonlin1} we compare the order of convergence in  $L^{1}$ norm of the Scharfetter-Gummel extended scheme (\ref{fluxSG1}) and of our first and second order fully upwind fluxes (\ref{fluxFU1})-(\ref{fluxFU2}). It appears that the Scharfetter-Gummel scheme is second order accurate, as expected since the diffusion is not degenerate. Moreover, we verify experimentally that our scheme (\ref{fluxFU2}) is second-order accurate and we notice that the $L^{1}$ error obtained with it is smaller than that obtained with the Scharfetter-Gummel extended scheme.

\begin{table}[!ht]
\centering
\begin{tabular}{|c|c|c|c|c|c|c|c|c|}
\hline $N_{x}$  & $L^{1}$ error  & Order & $L^{1}$ error & Order & $L^{1}$ error & Order \\ 
   & \textbf{SGext}  &  & \textbf{FU1} &  & \textbf{FU2} & \\ 
\hline 
       100     & $ 1.451.10^{-4} $  & 2 &  $ 2.667.10^{-3} $ & 0.87 & $ 8.237.10^{-5} $  & 1.87 \\ 
       200     & $ 3.619.10^{-5} $  & 2 &  $ 1.398.10^{-3} $ & 0.93 & $ 2.208.10^{-5} $  & 1.9  \\ 
       400     & $ 9.027.10^{-6} $  & 2 &  $ 7.156.10^{-4} $ & 0.97 & $ 5.778.10^{-6} $  & 1.93 \\
       800     & $ 2.251.10^{-6} $  & 2 &  $ 3.621.10^{-4} $ & 0.98 & $ 1.485.10^{-6} $  & 1.96 \\
       1600    & $ 5.614.10^{-7} $  & 2 &  $ 1.822.10^{-4} $ & 0.99 & $ 3.772.10^{-7} $  & 1.98 \\
\hline
\end{tabular} 
\caption{Example 1 - Experimental spatial order of convergence in $L^{1}$ norm.}
\label{tablenonlin1}
\end{table}

\paragraph*{Example 2 (Degenerate case).} We still consider the same test case, but now with
\begin{equation*} 
r(s)=\left\{\begin{array}{ll} (s-1)^{3} &\text{ if }  s \geq 1, \\ 0 & \text{ elsewhere,}  \end{array}\right. 
\end{equation*}
then $r'(s)=0$ for all $s \in (0,1)$. The initial data is
\begin{equation*}
u_{0}(x)= 1+0.5\sin(\pi x) \quad x \in (-1,1),
\end{equation*} 
and the final time is $T=0.01$. The diffusion vanishes in $\lbrace x \in (-1,1) \,: \, u(x) \leq 1 \rbrace$, which is not empty, then this test case is degenerate. In Figure \ref{ex2_1}, we represent the evolution of the deviation from the initial data of the approximate solution computed on a fine mesh made of 3200 cells with the scheme \eqref{fluxFU2}. We observe a loss of regularity during the evolution. 
\begin{figure}[!ht]
\centering
\includegraphics[width=2.6in]{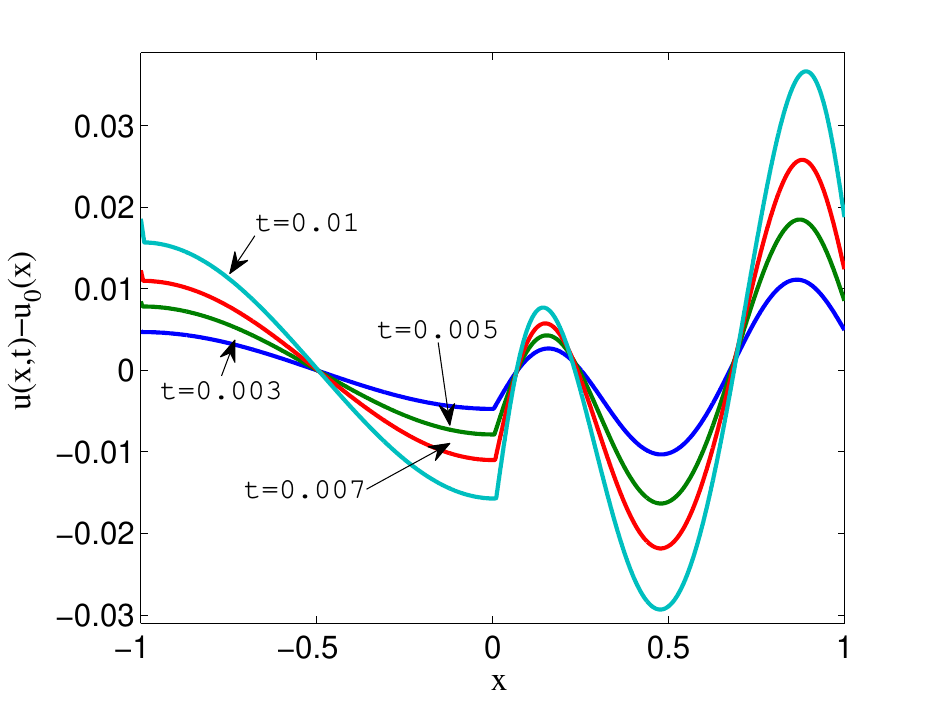}
\caption{Example 2 - Evolution of the deviation from the initial data $u(t)-u_{0}$.}
\label{ex2_1}
\end{figure}
In Table \ref{tabledeg1} we compare the order of convergence in  $L^{1}$ norm of the Scharfetter-Gummel extended scheme (\ref{fluxSG1}) and of our first and second order fully upwind fluxes (\ref{fluxFU1})-(\ref{fluxFU2}). In this case where $r'$ vanishes on a whole interval, it appears that the second-order scheme (\ref{fluxFU2}) is more accurate than the two others schemes. The Scharfetter-Gummel extended scheme is only one order accurate while second-order accuracy is almost preserved with our new scheme, in spite of the loss of regularity of the solution observed in Figure \ref{ex2_1}.

\begin{table}[!ht]
\centering
\begin{tabular}{|c|c|c|c|c|c|c|c|c|}
\hline $N_{x}$  & $L^{1}$ error  & Order & $L^{1}$ error & Order & $L^{1}$ error & Order \\ 
   & \textbf{SGext}  &  & \textbf{FU1} &  & \textbf{FU2} & \\ 
\hline 
       100   & $ 3.074.10^{-4} $  & 0.96 &  $ 2.697.10^{-4} $ & 0.55 & $ 1.053.10^{-4} $  & 1.83  \\ 
       200   & $ 1.554.10^{-4} $  & 0.98 &  $ 1.531.10^{-4} $ & 0.82 & $ 2.830.10^{-5} $  & 1.90  \\ 
       400   & $ 7.834.10^{-5} $  & 0.99 &  $ 8.096.10^{-5} $ & 0.92 & $ 8.040.10^{-6} $  & 1.82  \\
       800   & $ 3.928.10^{-5} $  & 1    &  $ 4.163.10^{-5} $ & 0.96 & $ 2.288.10^{-6} $  & 1.81  \\
       1600  & $ 1.966.10^{-5} $  & 1    &  $ 2.111.10^{-5} $ & 0.98 & $ 6.576.10^{-7} $  & 1.80  \\
\hline
\end{tabular} 
\caption{Example 2 - Experimental spatial order of convergence in $L^{1}$ norm.}
\label{tabledeg1}
\end{table}

\paragraph*{Example 3 (Degenerate case).} Finally we consider the equation \eqref{eqgenediff} on $(0,T)\times \Omega=(0,1/2)\times(0,1)$ with $r(s)=\max(u-1,0)$ and $\pa_{x}V=-1$. The initial data is $u_{0}(x)=0$ and we consider the following Dirichlet boundary conditions:
\begin{equation*}
\left\{\begin{array}{lcl} u(t,0) &=&e^{2t}  \\ u(t,1)&=&0  \end{array}, \quad t \in (0,T).\right.
\end{equation*}
The exact solution is then given by
\begin{equation*}
u(t,x)=\left\{\begin{array}{lcl} \exp(2t-x) & \text{ if } & x<2t,\\ 0 & \text{ if }& x>2t. \end{array}\right.
\end{equation*}
We compute the solution up to $t=0.3$ with $ \Delta t=10^{-4}$ and $N_{x}=40$ uniform cells. The results are shown in Figure \ref{ex3}. This example works well and it illustrates the advantage of using a high-order method even in the case of a discontinuous solution, since the shock is less diffused with our scheme \eqref{fluxFU2} than with the three others.
\begin{figure}[!ht]
\centering
\subfigure{\includegraphics[width=2.6in]{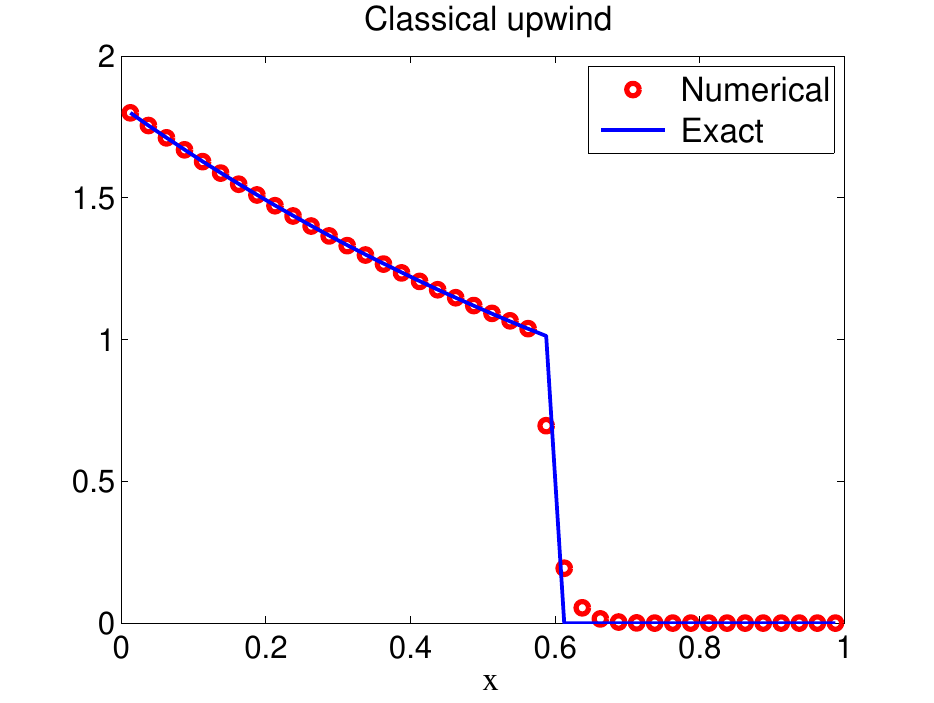}}
\subfigure{\includegraphics[width=2.6in]{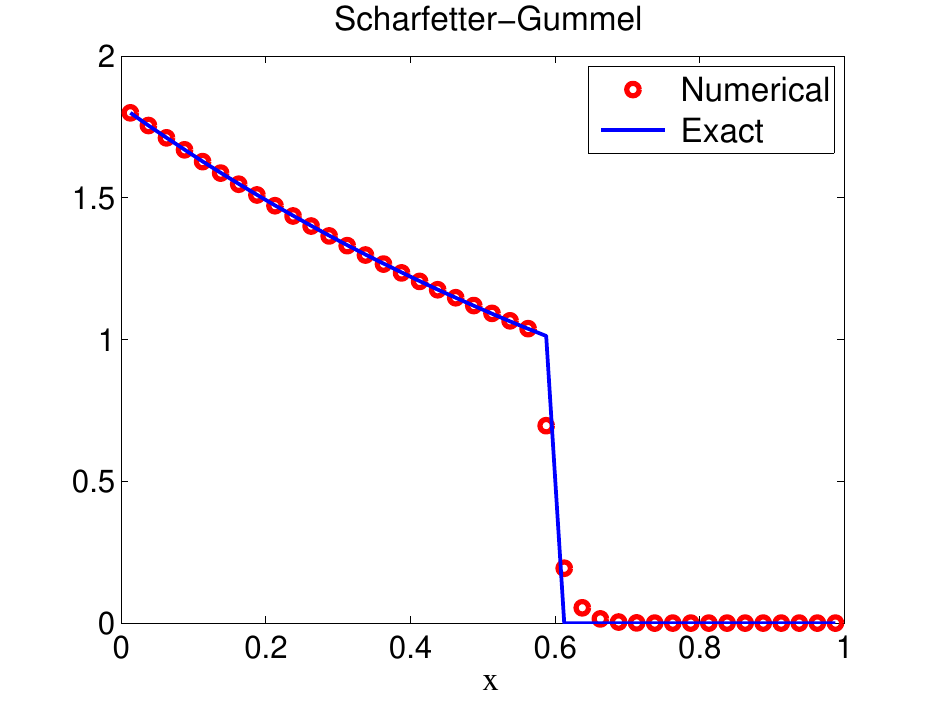}}
\subfigure{\includegraphics[width=2.6in]{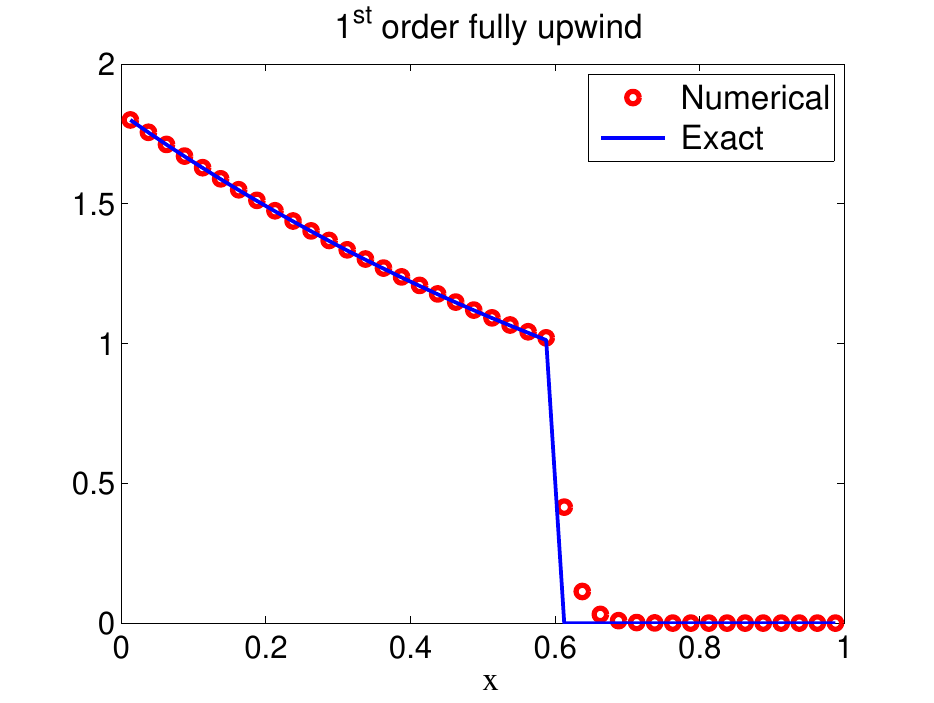}}
\subfigure{\includegraphics[width=2.6in]{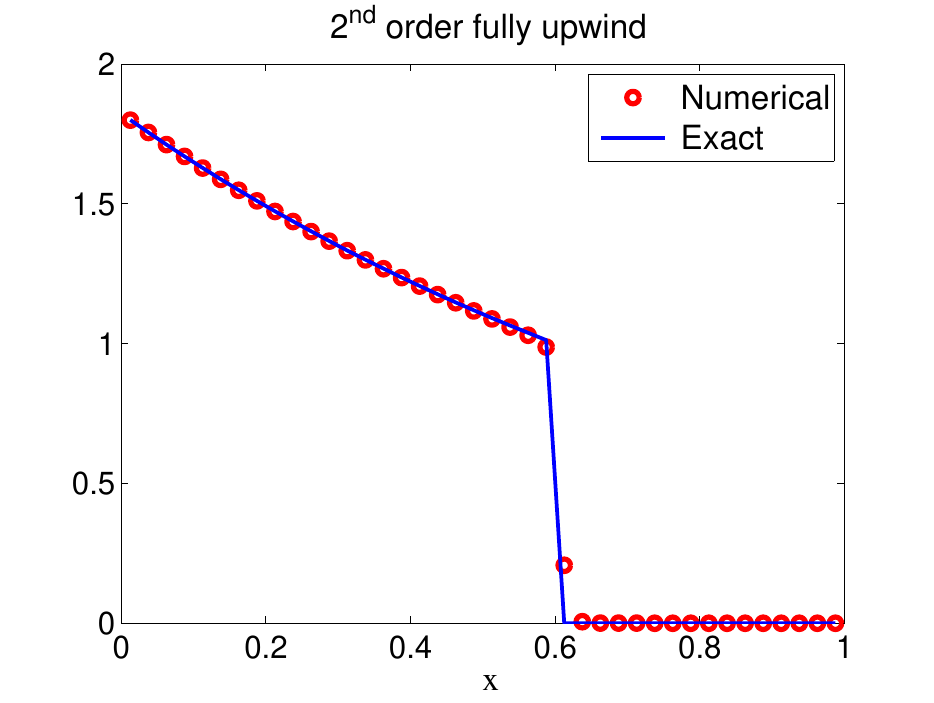}}
\caption{Example 3 - Numerical and exact solution computed at $t=0.3$ with different schemes.}
\label{ex3}
\end{figure}


\subsection{The drift-diffusion system for semiconductors}

We now consider the drift-diffusion system for semiconductors (\ref{DD}). In the two following examples, the Dirichlet boundary conditions satisfy (\ref{compatibility1})-(\ref{compatibility2}), so the thermal equilibrium is uniquely defined by (\ref{eqthermiqueDD}). We compute an approximation $(N^{eq}_{i},P^{eq}_{i},V^{eq}_{i})_{i=1,...,N_{x}}$ of this equilibrium with the finite volume scheme proposed by C. Chainais-Hillairet and F. Filbet in \cite{Chainais-Hillairet2007}.

\paragraph*{Example 4.} Firstly we consider a 1D test case on $\Omega=(0,1)$. We take $r(s)=s^{2}$. Initial data are
\begin{equation*}
N_{0}(x)=\left\{\begin{array}{ccc} 0 & \text{ for }& x \leq 0.5 \\ 1 & \text{ for } & x>0.5 \end{array}\right. , \quad P_{0}(x)=\left\{\begin{array}{ccc} 1 & \text{ for }& x \leq 0.5 \\ 0 & \text{ for } & x>0.5 \end{array}\right.,
\end{equation*}
and we consider the following Dirichlet boundary conditions
\begin{equation*}
\begin{array}{ccl} N(0,t)=0, \quad & P(0,t) = 1, \quad & V(0,t)=-1,\\ N(1,t) = 1, \quad & P(1,t)=0, \quad & V(1,t)=1. \end{array}
\end{equation*}
The doping profile is
\begin{equation*}
C(x)=\left\{\begin{array}{ccl} -1 & \text{ for } & x \leq 0.5, \\ +1  & \text{ for } & x > 0.5. \end{array}\right.
\end{equation*}
The time step is $\Delta t = 5.10^{-5}$ and the final time $T=10$. The domain $(0,1)$ is divided into $N_{x}= 64$ uniform cells.\\
In Figure \ref{ex3_1}, we compare the discrete relative energy $\mathcal{E}_\Delta(t^{n})$ and its dissipation $\mathcal{I}_\Delta(t^{n})$ obtained with the Scharfetter-Gummel extended scheme (\ref{fluxSG1}), the classical upwind scheme (\ref{fluxCU}) and our first and second order schemes (\ref{fluxFU1})-(\ref{fluxFU2}). The classical upwind flux (\ref{fluxCU}) does not preserve the thermal equilibrium, which explains the phenomenon of saturation observed with it. The Scharfetter-Gummel extended flux (\ref{fluxSG1}) preserves the equilibrium at the points where the densities $N$ and $P$ do not vanish, but due to the zero boundary conditions on the left for $N$ and on the right for $P$, there is also a phenomenon of saturation with it. Contrary to these two schemes, our new schemes (\ref{fluxFU1})-(\ref{fluxFU2}) which preserve the equilibrium everywhere, provide a satisfying long-time behavior. Moreover, we computed the relative energy and its dissipation with our schemes for different numbers $N_{x}$ of cells and notice that the decay rate does not depend on the spatial step size. We obtained satisfying results even for few number of cells.

\begin{figure}[!ht]
\centering
\subfigure{\includegraphics[width=2.6in]{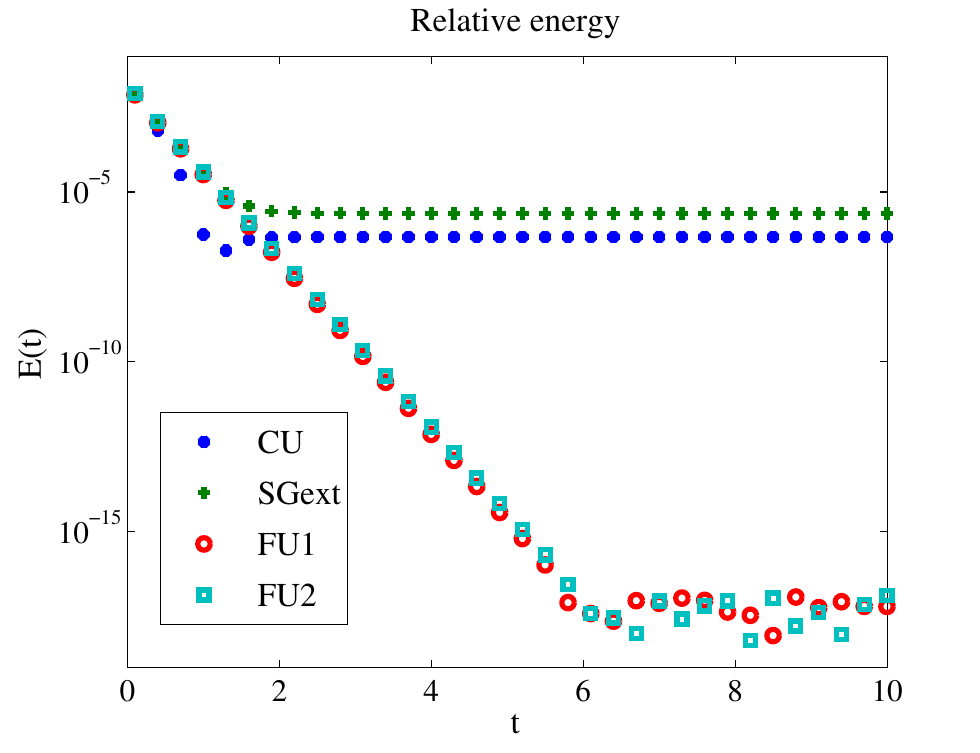}}
\subfigure{\includegraphics[width=2.6in]{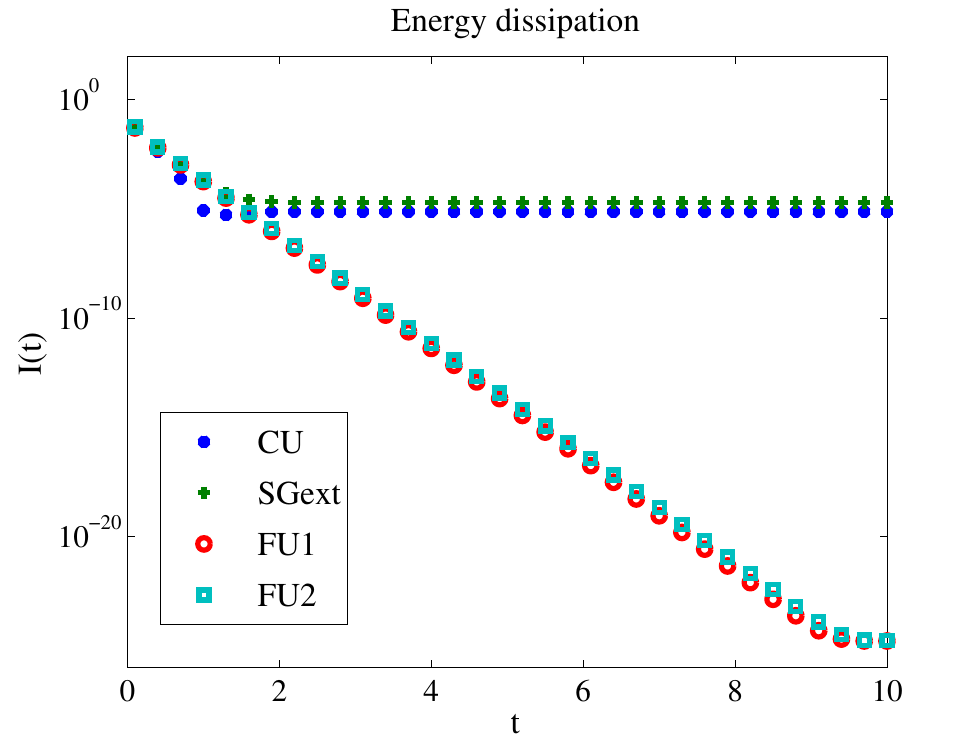}}
\caption{Example 4 - Evolution of the relative energy $\mathcal{E}_\Delta (t^{n})$ and its dissipation $\mathcal{I}_\Delta(t^{n})$ in log-scale for different schemes ($N_{x}=64$).}
\label{ex3_1}
\end{figure}

\paragraph*{Example 5.} Let us consider now a 2D test case picked on the paper of C. Chainais-Hillairet, J. G. Liu and Y. J. Peng \cite{Chainais-Hillairet2003}. As in the previous example, the Dirichlet boundary conditions vanish on some part of the boundary. The time step is $\Delta t= 10^{-4}$, the final time is $T=10$ and we compute an approximate solution on a $32 \times 32$ Cartesian grid.\\
In Figure \ref{ex4}, we compare the discrete relative energy $\mathcal{E}_\Delta(t^{n})$ and its dissipation $\mathcal{I}_\Delta(t^{n})$ obtained with the Scharfetter-Gummel extended scheme (\ref{fluxSG1}), the classical upwind scheme (\ref{fluxCU}) and the fully upwind schemes (\ref{fluxFU1})-(\ref{fluxFU2}). We make the same observations as in Example 4: there is a phenomenon of saturation with the Scharfetter-Gummel extended and the classical upwind schemes, and not with our new scheme. Moreover, the decay rate does not depend on the number of grid cells chosen. 

\begin{figure}[!ht]
\centering
\subfigure{\includegraphics[width=2.6in]{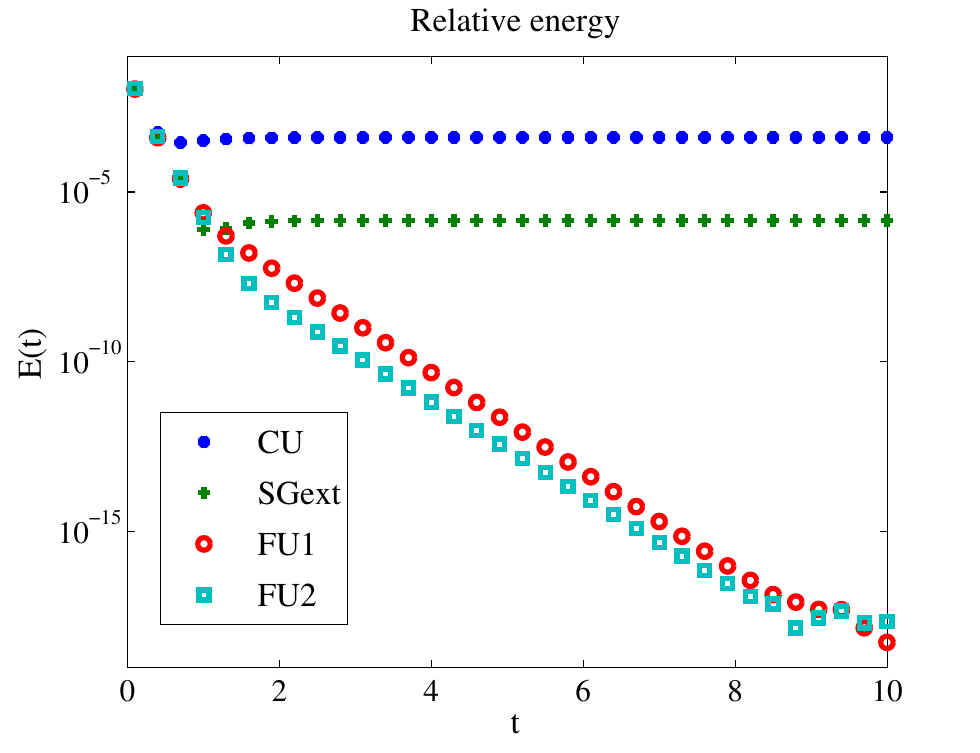}}
\subfigure{\includegraphics[width=2.6in]{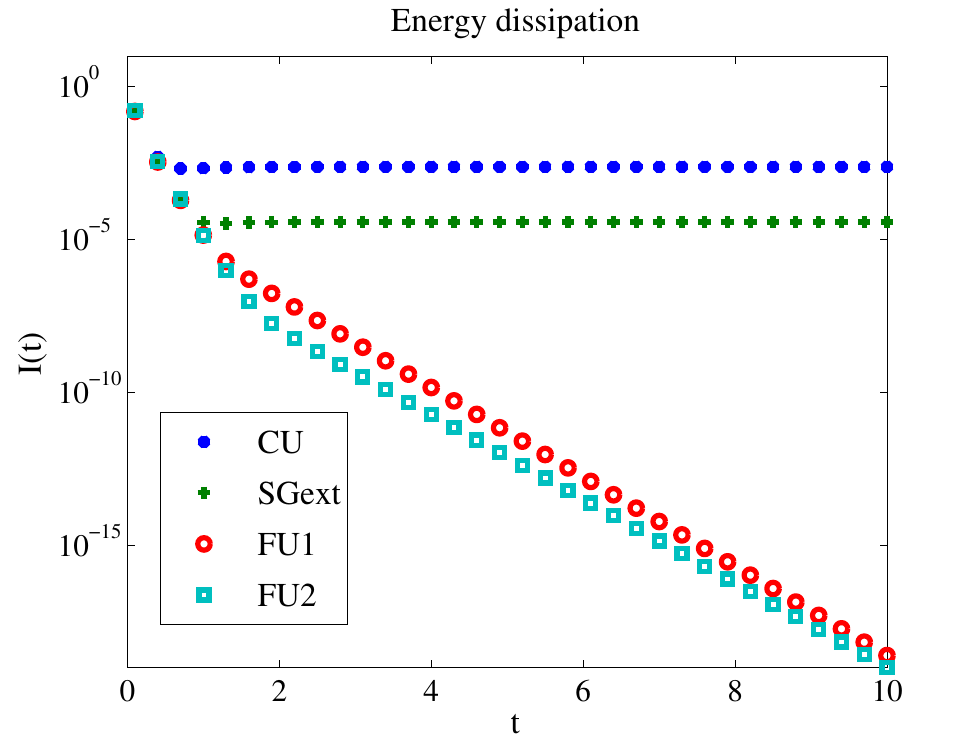}}
\caption{Example 5 - Evolution of the relative energy $\mathcal{E}_\Delta(t^{n})$ and its dissipation $\mathcal{I}_\Delta(t^{n})$ in log-scale for different schemes.}
\label{ex4}
\end{figure}


\subsection{The porous media equation}

In this part we approximate solutions to the porous media equation 
\begin{equation*}
\pa_{t}u=\nabla \cdot (xu+\nabla u^{m}).
\end{equation*}
We define an approximation $\left(U^{eq}_{i}\right)_{i=1,...,N_{x}}$ of the unique stationary solution $u^{eq}$ (\ref{barenblatt}) by
\begin{equation*}
U^{eq}_{i}=\left(\overline{C}-\frac{m -1}{2m}\left\vert x_{i}\right\vert^{2}\right)^{1/(m -1)}_{+}, \,\ i=1,...,N_{x},
\end{equation*}
where $\overline{C}$ is such that the discrete mass of $\left(U^{eq}_{i}\right)_{i =1,...,N_{x}}$ is equal to that of $\left(U^{0}_{i}\right)_{i=1,...,N_{x}}$, namely\\
$\ds{\sum_{i}\Delta x_{i}U^{eq}_{i}=\sum_{i}\Delta x_{i}U^{0}_{i}}$. We use a fixed point algorithm to compute this constant $ \overline{C}$.

\paragraph*{Example 6.} We consider the following one dimensional test case: $m=5$, with initial condition
\begin{equation*}
u_{0}(x)=\left\{\begin{array}{ll} 1 &\text{ if }\,  x \in (-3.7,-0.7) \cup (0.7,3.7), \\ 0 & \text{ otherwise. }  
\end{array}\right.
\end{equation*}
Then we compute the approximate solution on $(-5.5,5.5)$, which is divided into $N_{x}=160$ uniform cells. The time step is fixed to $\Delta t=10^{-4}$ and the final time is $T=10$.\\ 
In Figure \ref{ex5_1} we compare the discrete relative entropy $\mathcal{E}_\Delta(t^{n})$ and its dissipation $\mathcal{I}_\Delta(t^{n})$ obtained with the Scharfetter-Gummel extended scheme, the classical upwind scheme and the first and second order fully upwind schemes. We obtain almost the same behavior for the Scharfetter-Gummel scheme and the fully upwind schemes. We only notice that the dissipation $ \mathcal{I}_\Delta(t^{n})$ obtained with the Scharfetter-Gummel scheme saturates before those obtained with the fully upwind schemes. This phenomenon of saturation is still greater for the classical upwind scheme. Moreover, we observe an exponential decay of $\mathcal{E}_\Delta(t^{n})$ and $\mathcal{I}_\Delta(t^{n})$, at a rate -12. In their paper \cite{Carrillo2000}, J. A. Carrillo and G. Toscani obtain the following equation for the entropy dissipation:
\begin{equation*}
\frac{d}{dt}\mathcal{I}(t)=-2\,\mathcal{I}(t)-\mathcal{R}(t),
\end{equation*}
where $\mathcal{R}(t) \geq 0$ depends on the power $m$. Then they conclude with the exponential decay of the relative entropy $\mathcal{E}$ to zero at a rate -2. In our test where the initial condition is symmetric, we obtain a better rate, which seems to depend on $m$, taken equal to 5 here, thus it underlines the contribution of the term $ \mathcal{R}$ in this case. \\
However, if we now consider a nonsymmetric initial data $u_{0}(x)=\mathbf{1}_{[2,3]}(x)$ and compute the relative entropy $\mathcal{E}_\Delta(t^{n})$ obtained with our scheme \eqref{fluxFU2} for different values of $m$, we observe in Figure \ref{ex5_2} an exponential decay with rate -2, independently of the value of $m$. Thus in this case the estimate of decay of the relative entropy seems sharp.

\begin{figure}[!ht]
\centering
\subfigure{\includegraphics[width=2.6in]{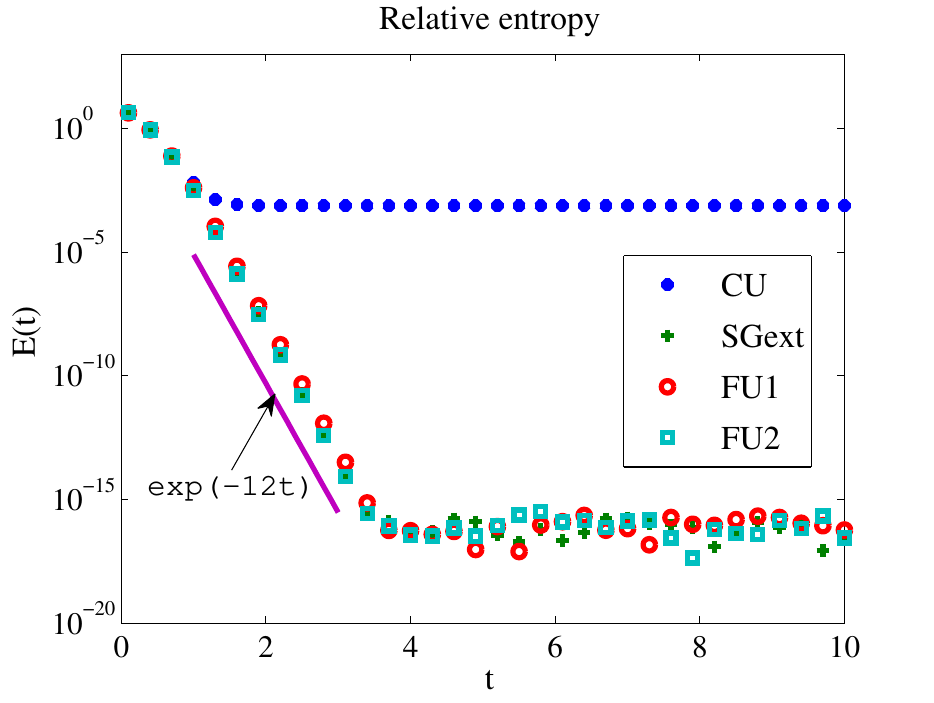}}
\subfigure{\includegraphics[width=2.6in]{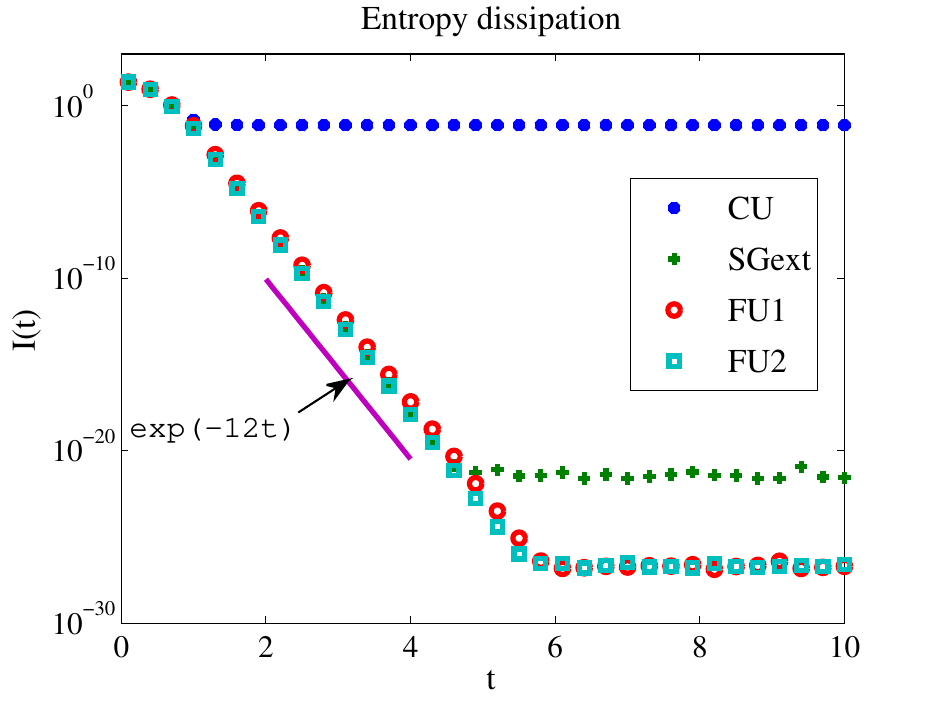}}
\caption{Example 6 - Evolution of the relative entropy $\mathcal{E}_\Delta(t^{n})$ and its dissipation $\mathcal{I}_\Delta(t^{n})$ in log-scale for different schemes.}
\label{ex5_1}
\end{figure}

\begin{figure}[!ht]
\centering
\includegraphics[width=2.6in]{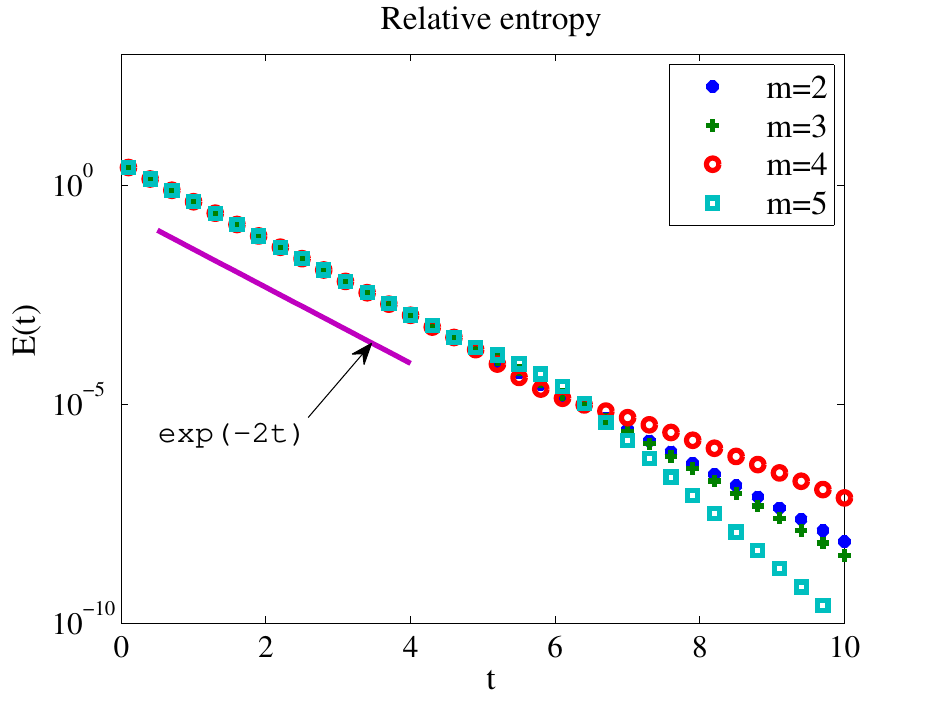}
\caption{Evolution of the relative entropy $\mathcal{E}_\Delta(t^{n})$ in log-scale for different values of $m$ in the case of a nonsymmetric initial data.}
\label{ex5_2}
\end{figure}

\paragraph*{Example 7.} We still consider the porous media equation, but now in two space dimension on $\Omega = (-10,10) \times (-10,10)$. We take $m=4$ and the initial condition is
\begin{equation*}
u_{0}(x,y)=\left\{\begin{array}{ll} 
\exp\left(-\frac{1}{6-(x-2)^{2}-(y+2)^{2}}\right) &\text{ if }\,  (x-2)^{2}+(y+2)^{2}<6, 
\\ 
\,
\\
\exp\left(-\frac{1}{6-(x+2)^{2}-(y-2)^{2}}\right) &\text{ if }\, (x+2)^{2}+(y-2)^{2}<6, 
\\
\,
\\
0 & \text{ otherwise. } 
\end{array}\right.
\end{equation*}
We compute the approximate solution on a $200 \times 200$ Cartesian grid, with $\Delta t=10^{-4}$ and $T=10$. \\
In Figure \ref{ex6_1} we compare the discrete relative entropy $\mathcal{E}_\Delta(t^{n})$ and its dissipation $\mathcal{I}_\Delta(t^{n})$ obtained with the Scharfetter-Gummel scheme, the classical upwind scheme and the fully upwind schemes, and obtain an exponential decay at a rate -4 with our new scheme \eqref{fluxFU2}.\\
Figure \ref{ex6_2} presents the evolution of the density of gas $u$ computed with our second-order scheme at four different times $t=0$, $t=0.5$, $t=1$ and $t=10$ and the approximation of the stationary solution $u^{eq}$ corresponding to this initial data.\\

\begin{figure}[!ht]
\centering
\subfigure{\includegraphics[width=2.6in]{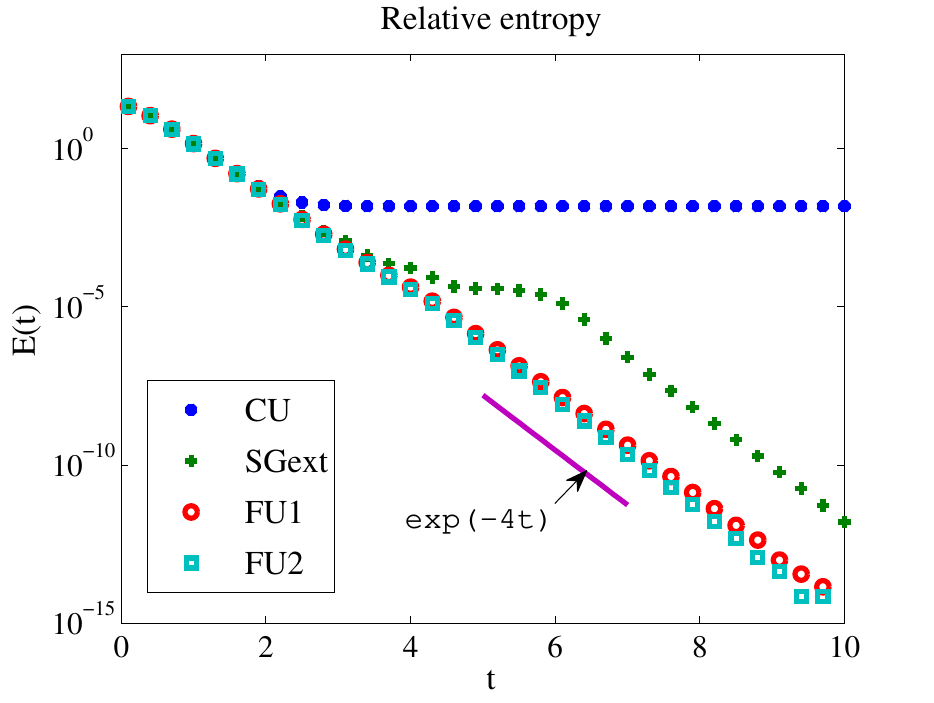}}
\subfigure{\includegraphics[width=2.6in]{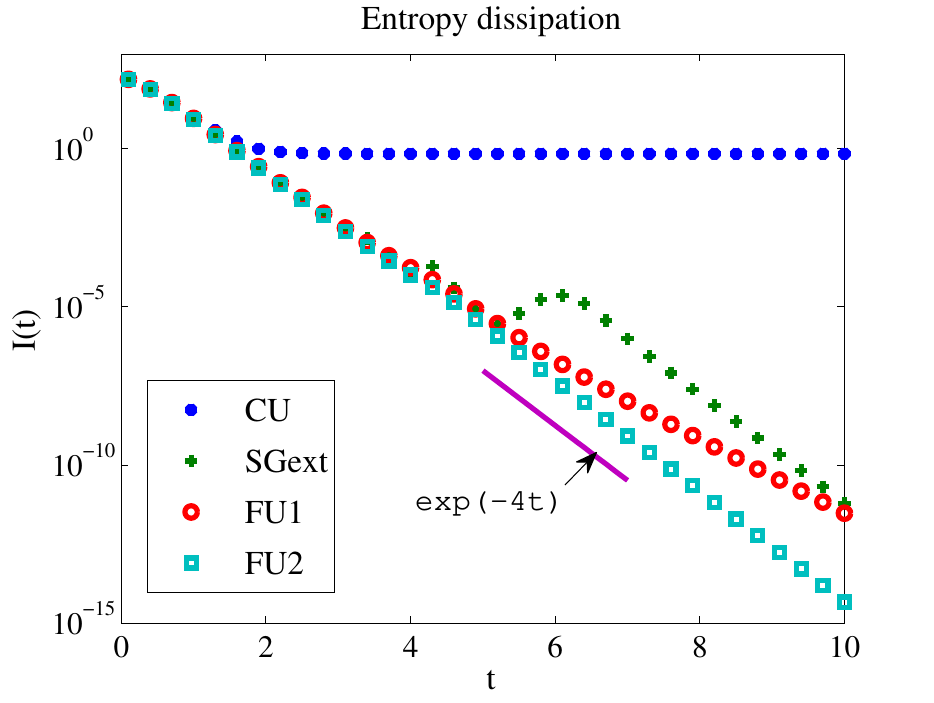}}
\caption{Example 7 - Evolution of the relative entropy $\mathcal{E}_\Delta(t^{n})$ and its dissipation $\mathcal{I}_\Delta(t^{n})$ in log-scale for different schemes.}
\label{ex6_1}
\end{figure}

\begin{figure}[!ht]
\centering
\subfigure[$t=0$]{\includegraphics[width=2.6in]{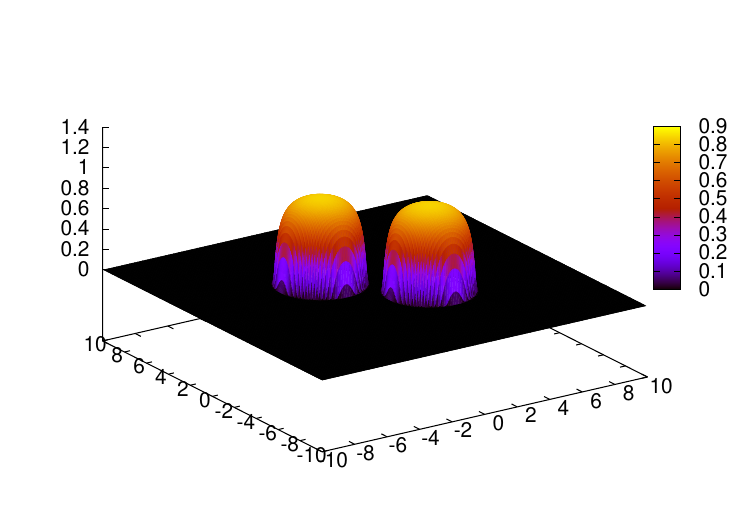}}
\subfigure[$t=0.5$]{\includegraphics[width=2.6in]{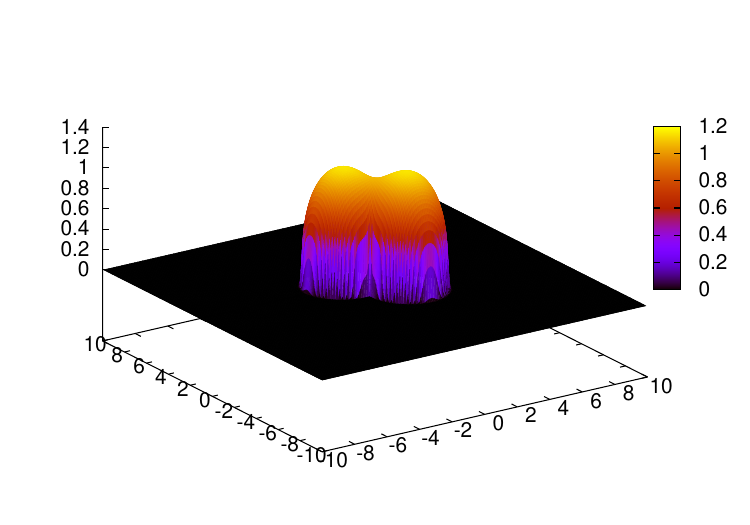}}
\subfigure[$t=1$]{\includegraphics[width=2.6in]{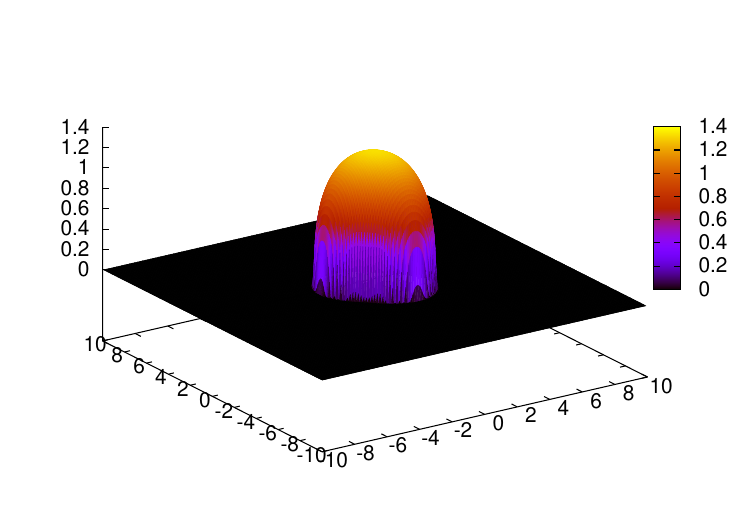}}
\subfigure[$t=10$]{\includegraphics[width=2.6in]{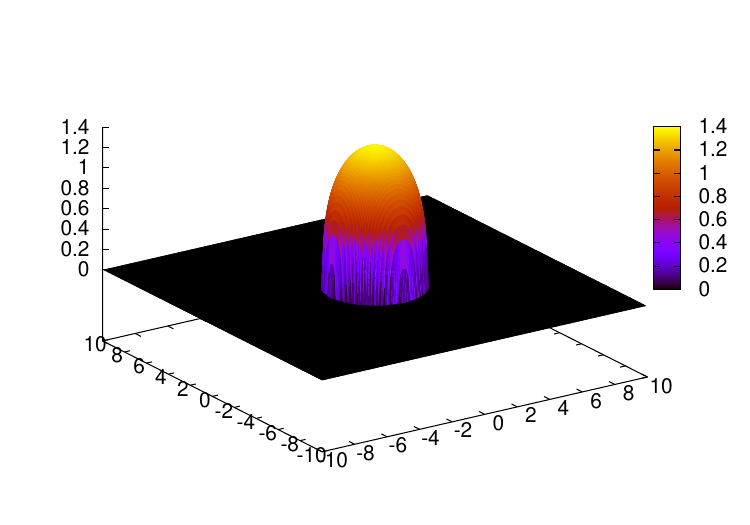}}
\subfigure[Stationary solution]{\includegraphics[scale=0.45]{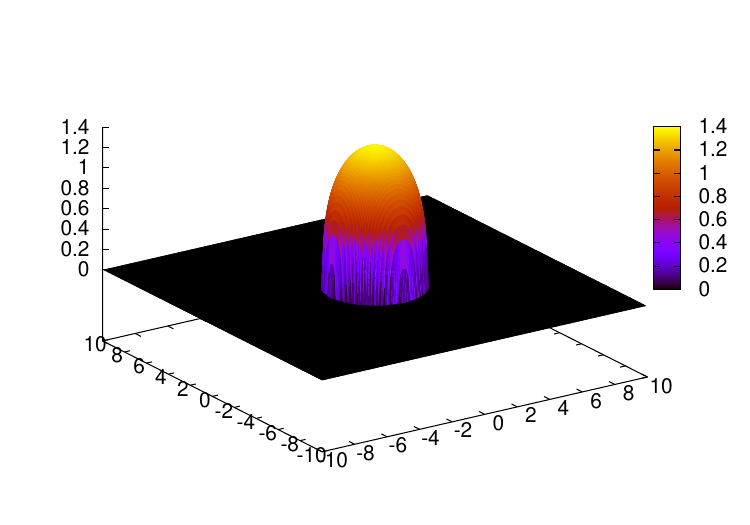}}
\caption{Example 7 - Evolution of the density of gas $u$ and corresponding stationary solution $u^{eq}$.}
\label{ex6_2}
\end{figure}


\subsection{Nonlinear Fokker-Planck equations for fermions and bosons}

\paragraph*{Example 8.}  We first consider the nonlinear Fokker-Planck equation (\ref{bosonfermion}) for fermions ($k=-1$). As in the porous media equation case, we define an approximation $\left(U^{eq}_{i}\right)_{i=1,...,N_{x}}$ of the unique stationary solution $u^{eq}$ (\ref{eqbosonfermion}) by
\begin{equation*}
U^{eq}_{i}=\frac{1}{\overline{\beta}e^{\frac{|x_{i}|^{2}}{2}}+1}, \,\ i=1,...,N_{x},
\end{equation*}
where $\overline{\beta} \geq 0$ is such that the discrete mass of $\left(U^{eq}_{i}\right)_{i =1,...,N_{x}}$ is equal to that of $\left(U^{0}_{i}\right)_{i=1,...,N_{x}}$. We use a fixed point algorithm to compute this constant $ \overline{\beta}$.\\
We consider a 3D test case. The initial condition is chosen as the sum of four Gaussian distributions:
\begin{equation*}
u_{0}(x)=\frac{1}{2\sqrt{2\pi}}\left(\exp \left(-\frac{|x-x_{1}|^{2}}{2}\right)+\exp \left(-\frac{|x-x_{2}|^{2}}{2}\right)+\exp \left(-\frac{|x-x_{3}|^{2}}{2}\right)+\exp \left(-\frac{|x-x_{4}|^{2}}{2}\right)\right),
\end{equation*}
where $x_{1}=(2,2,2)$, $x_{2}=(-2,-2,-2)$, $x_{3}=(2,-2,2)$ and $x_{4}=(-2,2,-2)$.\\
We consider a $40 \times 40 \times 40$ Cartesian grid of $\Omega=(-8,8)^{3}$, $\Delta t=10^{-4}$ and $T=10$.\\
Evolution of the discrete relative entropy $ \mathcal{E}_\Delta(t^{n})$, its dissipation $ \mathcal{I}_\Delta(t^{n})$ and $\Vert U^{n}-U^{eq}\Vert_{L^{1}}$ obtained with the scheme (\ref{fluxFU2}) is presented in Figure \ref{ex7_1}. We observe exponential decay rate of these quantities, which is in agreement with the result proved by J. A. Carrillo, Ph. Laurençot and J. Rosado in \cite{Carrillo2009}. \\
In Figure \ref{ex7_2} we report the evolution of the level set of the distribution function $u(t,x,y,z)=0.1$ at different times and the level set of the corresponding equilibrium solution $u^{eq}(x,y,z)=0.1$.

\begin{figure}[!ht]
\centering
\includegraphics[width=2.6in]{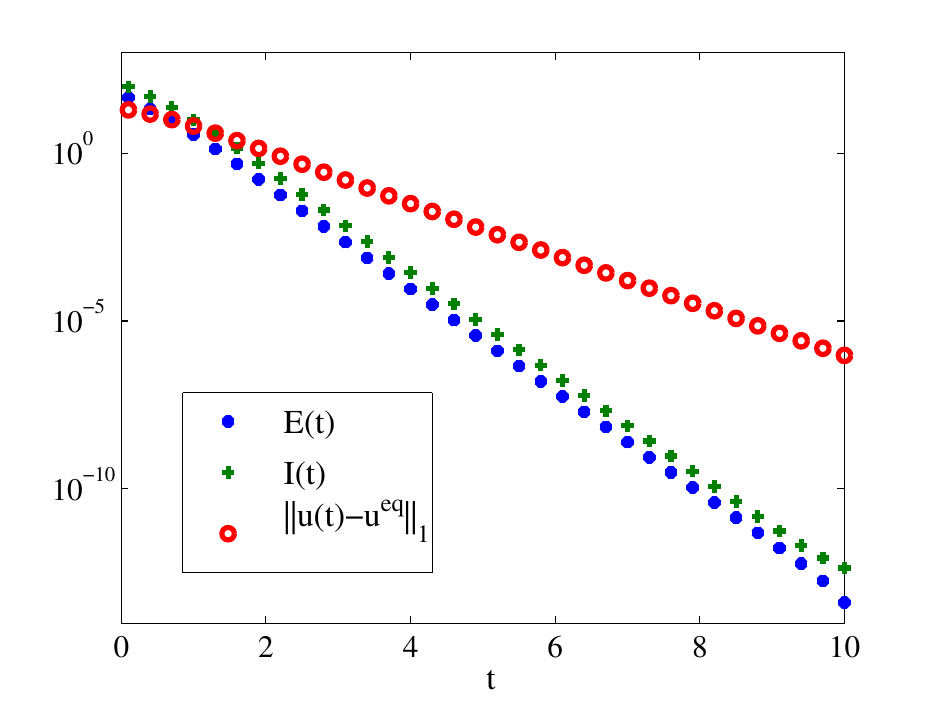}
\caption{Example 8 - Evolution of the relative entropy $\mathcal{E}_\Delta(t^{n})$, the dissipation $\mathcal{I}_\Delta(t^{n})$ and the $L^{1}$ norm $\Vert U^{n}-U^{eq}\Vert_{1}$.}
\label{ex7_1}
\end{figure}

\begin{figure}[!ht]
\centering
\subfigure[$t=0$]{\includegraphics[width=2.6in]{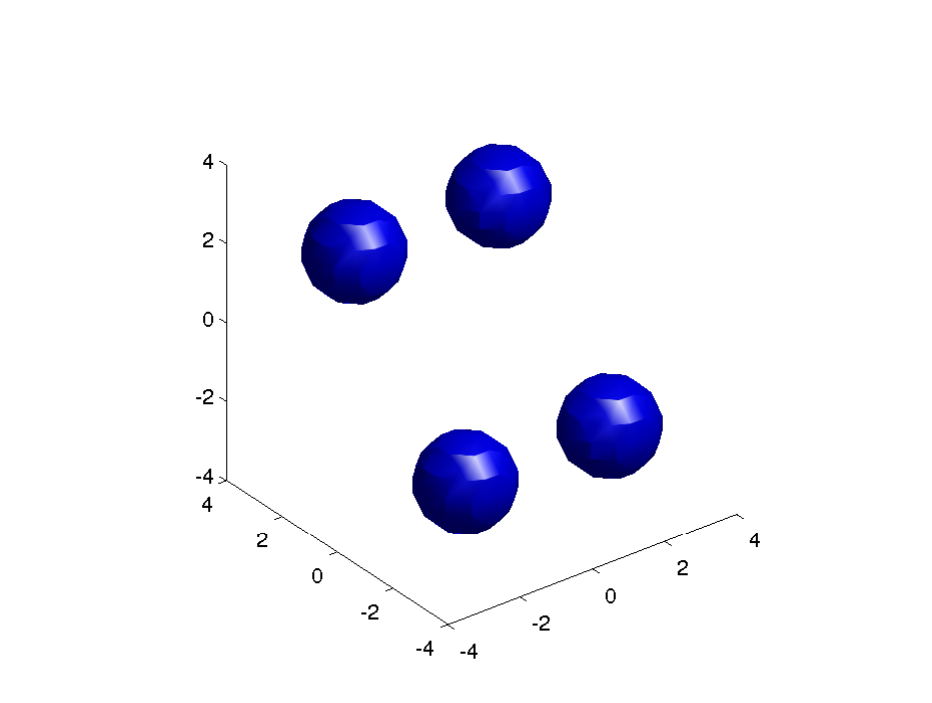}}
\subfigure[$t=0.2$]{\includegraphics[width=2.6in]{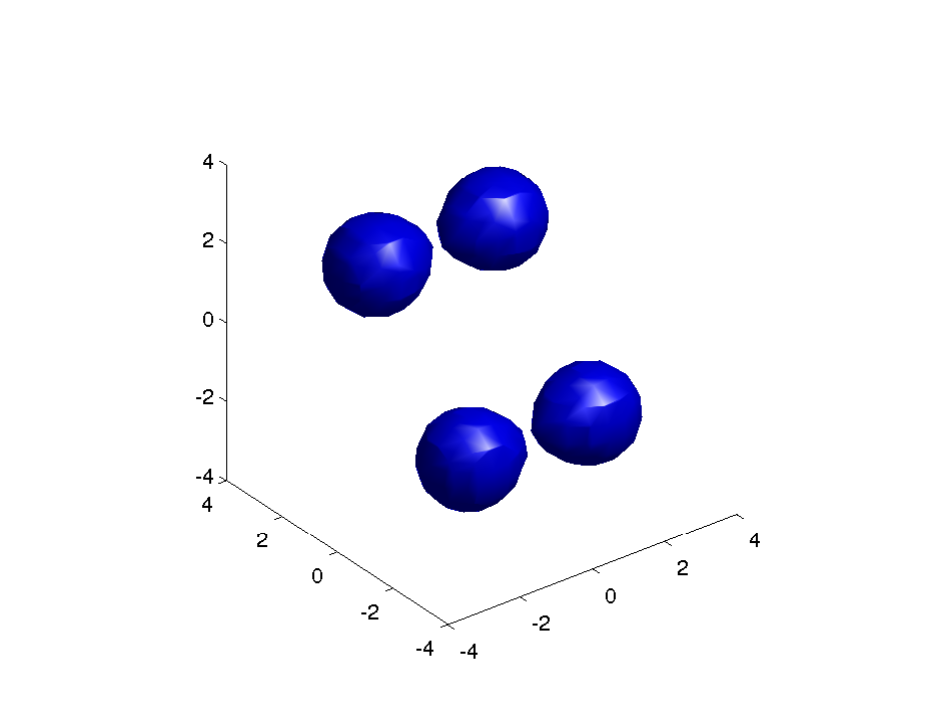}}
\subfigure[$t=0.4$]{\includegraphics[width=2.6in]{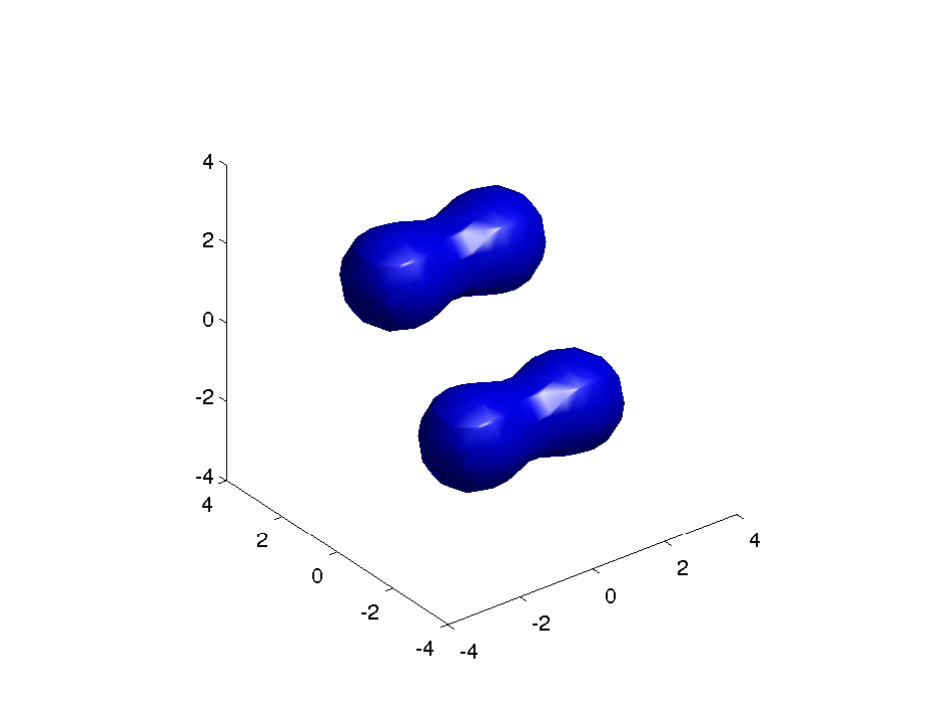}}
\subfigure[$t=1$]{\includegraphics[width=2.6in]{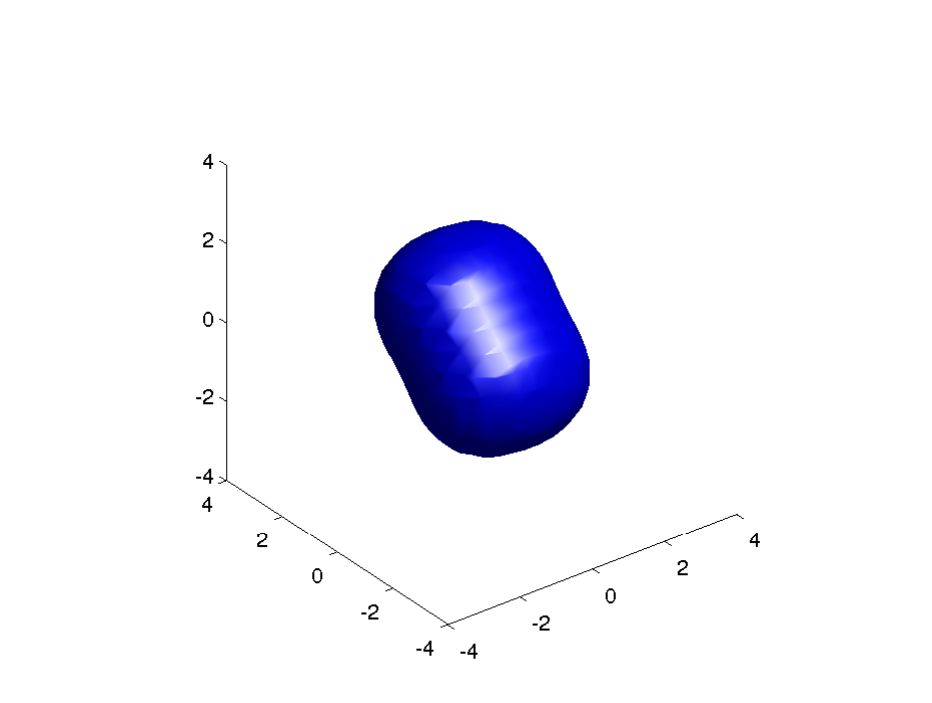}}
\subfigure[$t=10$]{\includegraphics[width=2.6in]{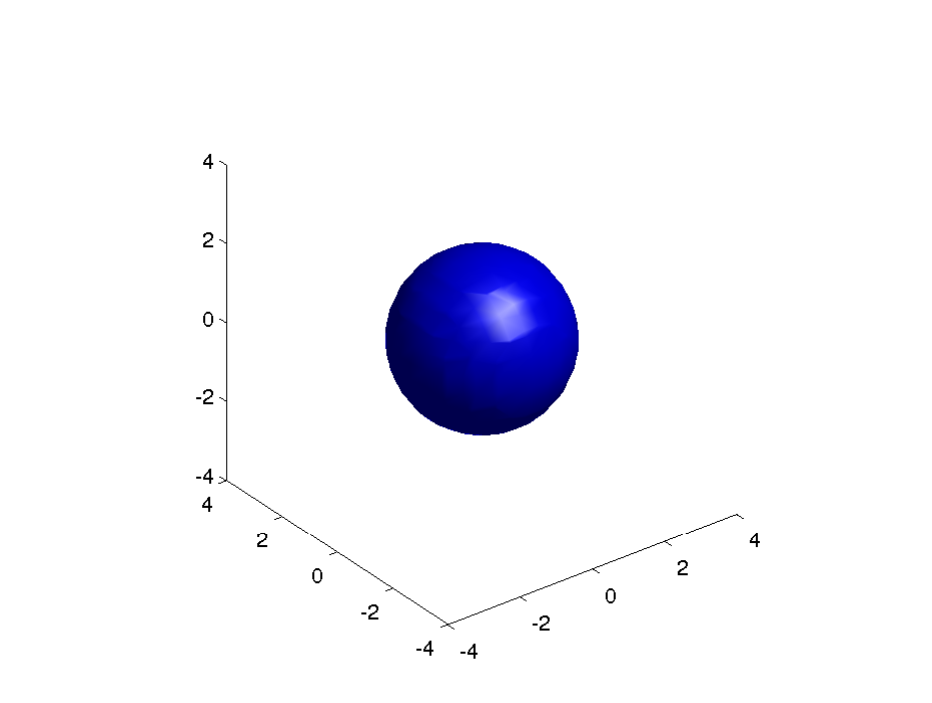}}
\subfigure[Stationary solution]{\includegraphics[width=2.6in]{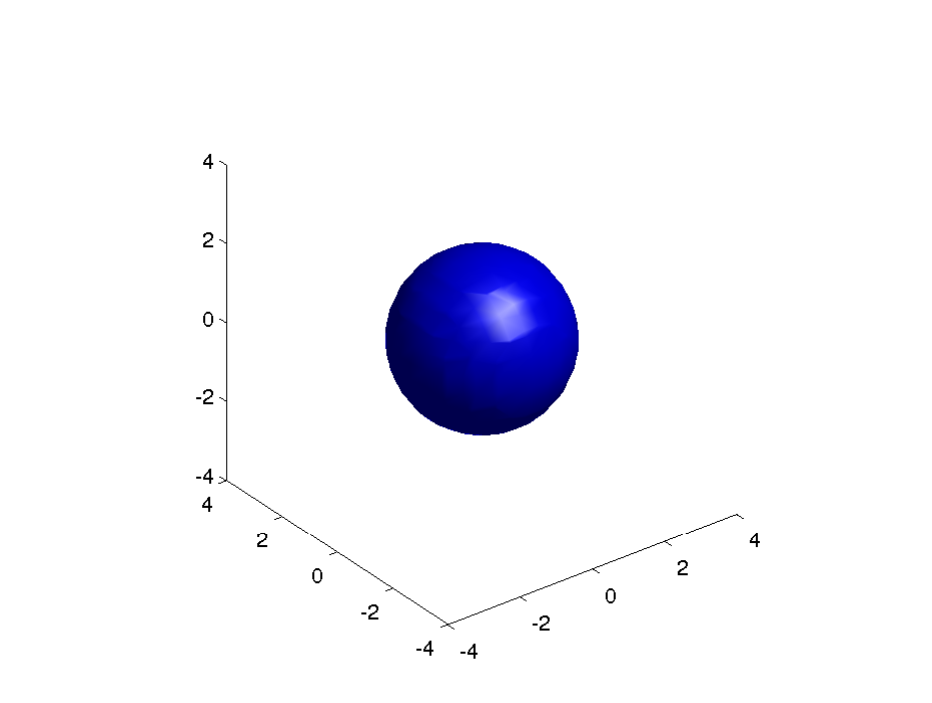}}
\caption{Example 8 - Evolution of the level set $u(t,x,y,z)=0.1$ and level set of the corresponding stationary solution $u^{eq}(x,y,z)=0.1$.}
\label{ex7_2}
\end{figure}

\paragraph*{Example 9.} We now consider the more general Fokker-Planck equation \eqref{eqbosons.gene} with $N=3$ in 1D:
\begin{equation*}
\pa_{t}u=\pa_{x}(xu(1+u^{3})+\pa_{x}u).
\end{equation*}
The initial condition is given by the sum of two Gaussian distributions:
\begin{equation*}
u_{0}(x)=\frac{M}{2\sqrt{2\pi}}\left(\exp\left(-\frac{|x-2|^{2}}{2}\right)+\exp\left(-\frac{|x+2|^{2}}{2}\right)\right),
\end{equation*}
where $M \geq 0$ is the mass of $u_{0}$. We compute an approximate solution with the scheme \eqref{fluxFU2} for two different values of $M$. The computational domain $(-10,10)$ is divided into $N_{x}=500$ uniform cells. \\
According to the paper of N. Ben Abdallah, I. Gamba and G. Toscani \cite{BenAbdallah2011}, there is a phenomenon of critical mass in this case. In Figure \ref{ex9_bosons1}, we represent the evolution of the density $u$ until time $T=10$ for an initial sub-critical mass $M=1$. We observe the convergence of the solution to the unique minimizer $u^{eq}$ of the entropy functional, given by
\begin{equation*}
u^{eq}(x)=\left(\beta \, e^{3x^{2}/2}-1\right)^{-\frac{1}{3}},
\end{equation*}
according to \cite{BenAbdallah2011}, where $\beta$ is such that $\int u^{eq}(x)\,dx=M$. Moreover, we observe in this case an exponential decay rate of the dissipation and the $L^{1}$ distance between the solution and the equilibrium.\\
In Figure \ref{ex9_bosons3}, we represent the evolution of the density $u$ for an initial super-critical mass $M=10$ until time $T=0.9$. We observe in this case the convergence of the solution to an equilibrium which has a singular part localized in the origin, which is in agreement with the result proved in \cite{BenAbdallah2011}.

\begin{figure}[!ht]
\centering
\subfigure{\includegraphics[width=2.6in]{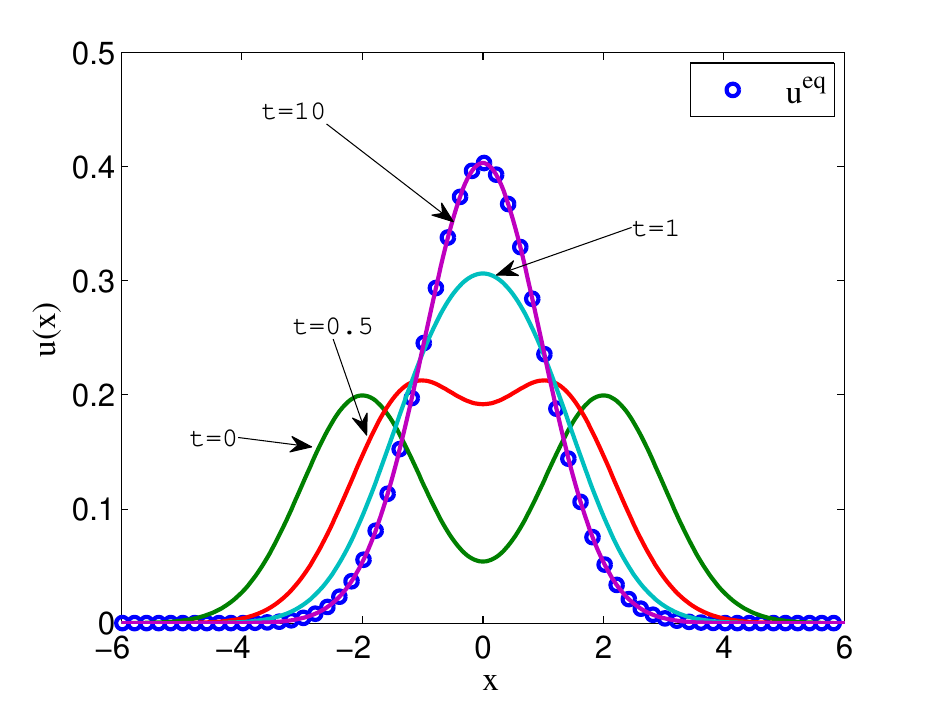}}
\subfigure{\includegraphics[width=2.6in]{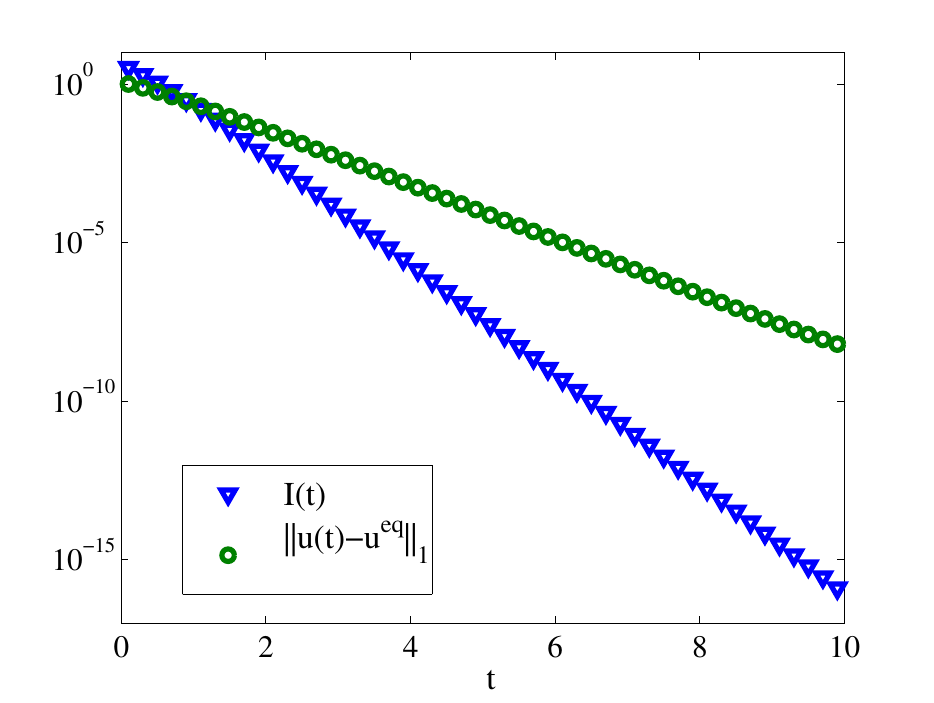}}
\caption{Example 9 - Evolution of the density $u$ of sub-critical mass $M=1$ (left) and of the corresponding dissipation $\mathcal{I}_\Delta(t^{n})$ and $L^{1}$ norm $\Vert U^{n}-U^{eq}\Vert_{1}$ (right).}
\label{ex9_bosons1}
\end{figure}

\begin{figure}[!ht]
\centering
\subfigure[$t=0$]{\includegraphics[width=1.6in]{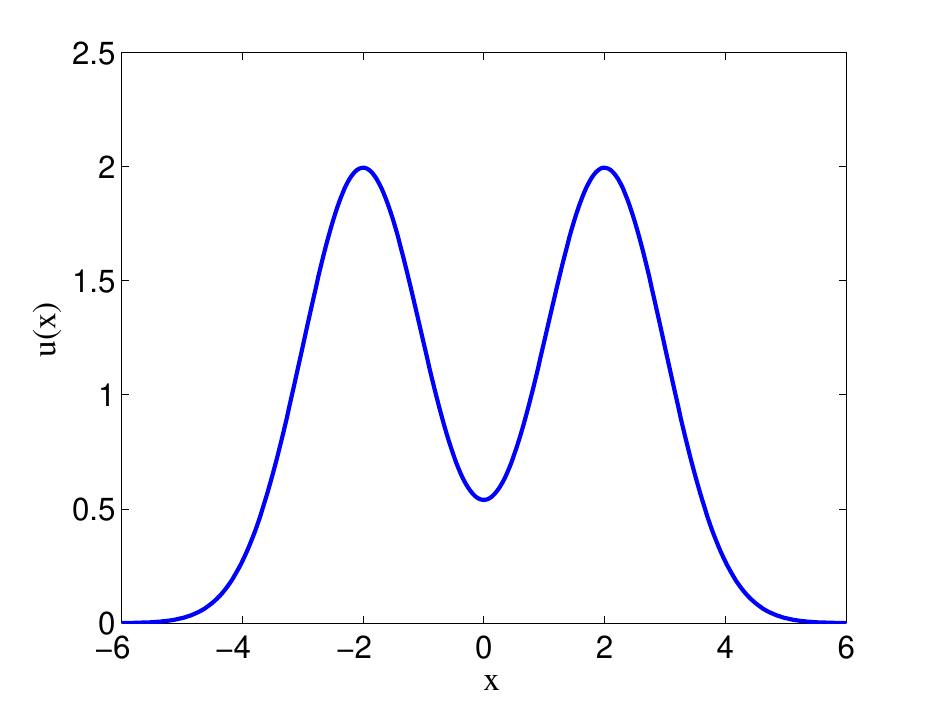}}
\subfigure[$t=0.05$]{\includegraphics[width=1.6in]{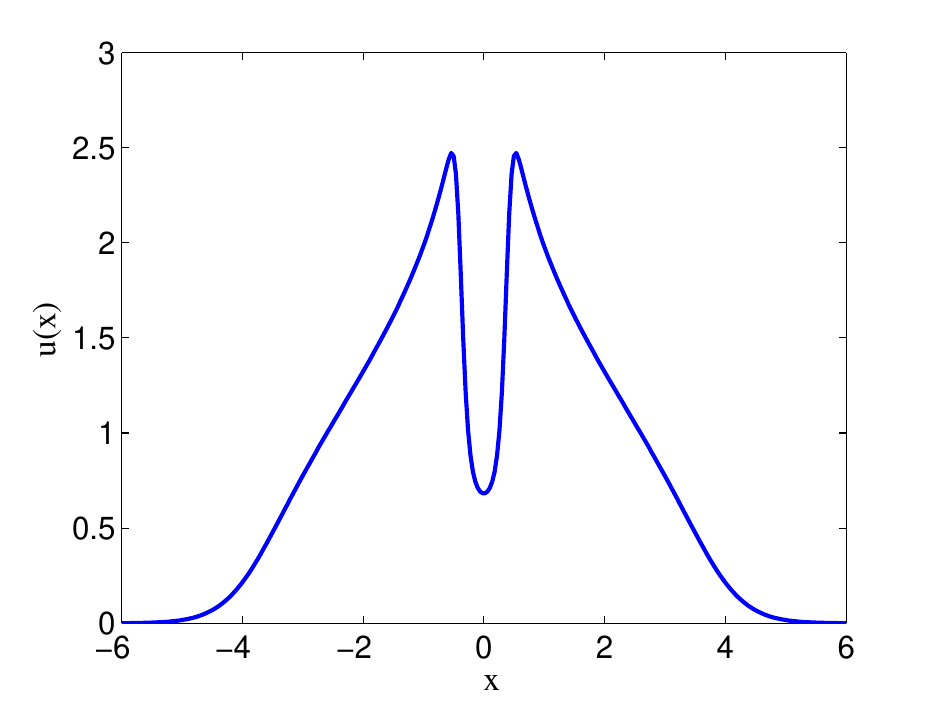}}
\subfigure[$t=0.2$]{\includegraphics[width=1.6in]{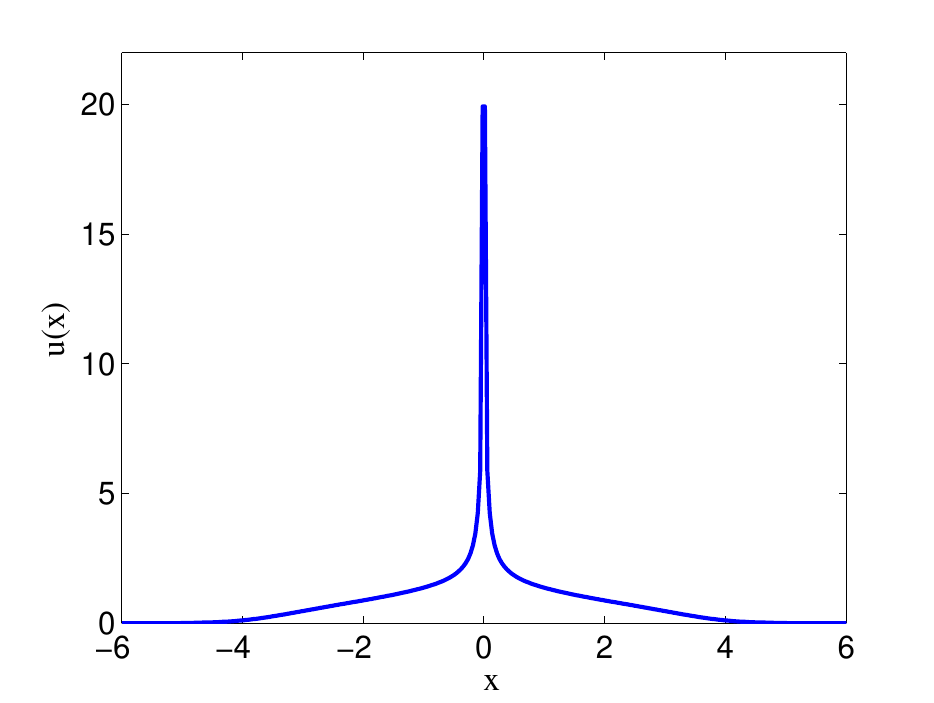}}
\subfigure[$t=0.9$]{\includegraphics[width=1.6in]{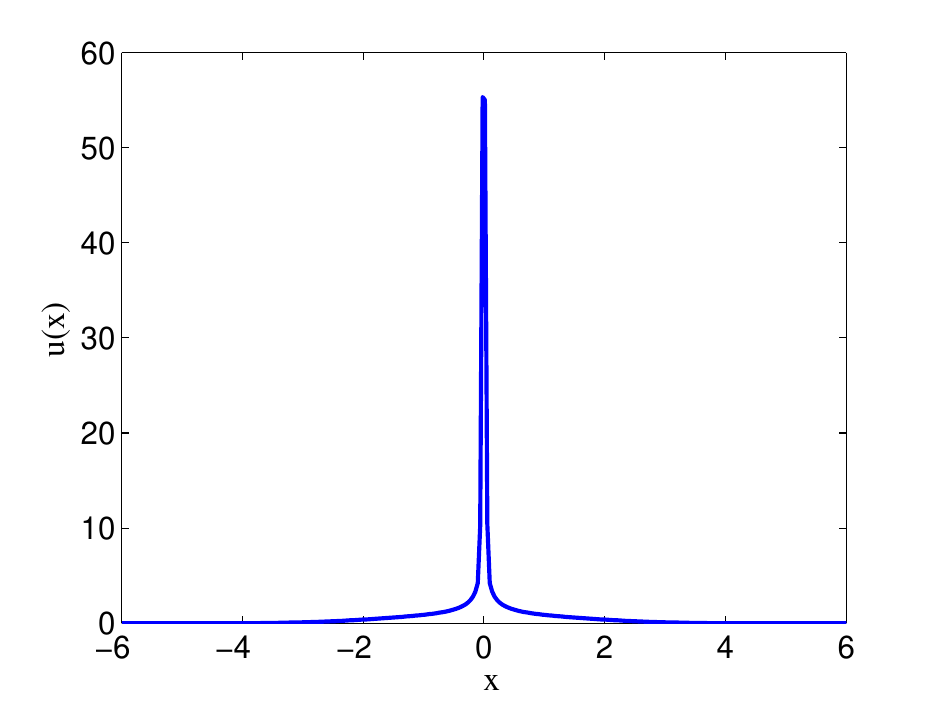}}
\caption{Example 9 - Evolution of the density $u$ of super-critical mass $M=10$.}
\label{ex9_bosons3}
\end{figure}

\subsection{The Buckley-Leverett equation}

Finally we consider the Buckley-Leverett equation, with both nonlinear convection and diffusion:
\begin{equation}
\pa_{t}u=\pa_{x}\left(-f(u)+\pa_{x}r(u)\right).
\label{BLeq}
\end{equation}
The Buckley-Leverett equation is a simple model for displacement of oil by water in oil reservoirs. The function $u(t,x) \in [0,1]$ represents the fraction of fluid corresponding to oil. We consider a fractional flow function $f$ with a s-shaped form
\begin{equation*}
f(u)=\frac{u^{2}}{u^{2}+(1-u)^{2}}.
\end{equation*}
This choice corresponds to a model which does not include the gravitational effects. The function $r$ is such that $r'(u)=\varepsilon \, \nu(u)$, where the capillary diffusion coefficient is given by
\begin{equation*}
\nu(u)=4u(1-u).
\end{equation*}
The scaling parameter $ \varepsilon>0$ in front of the capillary diffusion is usually small. \\
In this particular case, the Buckley-Leverett equation \eqref{BLeq} possesses a functional which dissipates a quantity. Indeed, rewriting the flux under the form \eqref{fluxgene} by taking $V=-x$ and 
\begin{equation*}
h(u)=4 \left(\log(u)-3u+2u^{2}-\frac{2}{3}u^{3}\right),
\end{equation*}
multiplying the equation \eqref{BLeq} by $\left(-x+h(u)\right)$ and integrating over $\Omega$, we get
\begin{equation*}
\frac{d}{dt}\int_{\Omega}\left(-x+h(u)\right)\,u\,dx=-\int_{\Omega} f(u)\left|\pa_{x}\left(-x+h(u)\right)\right|^{2}\,dx \leq 0,
\end{equation*}
since $f(u) \geq 0$ for all $u \in [0,1]$.

\paragraph*{Example 10.} We consider the following test case \cite{Kurganov2000,Liu2011}: the domain $\Omega$ is $(0,1)$, the initial condition 
\begin{equation*}
u_{0}(x)=\left\{\begin{array}{ccl} 1-3x & \text{ if } & 0 \leq x \leq \frac{1}{3}, 
\\
\,
\\ 0 & \text{ if } & \frac{1}{3}<x \leq 1, \end{array}\right.
\end{equation*}
and the boundary condition $u(0,t)=1$. The domain is divided into $N_{x}=100$ cells and the time step is $\Delta t=10^{-4}$. The numerical solution computed at different times for different values of $ \varepsilon$ is shown in Figure \ref{ex8_1}. The results compare well with those in \cite{Kurganov2000,Liu2011}. Moreover, our scheme remains valid for all values of $\varepsilon$, even $ \varepsilon=0$. In this case the fully upwind flux degenerates into the well-known local Lax-Friedrichs flux.

\begin{figure}[!ht]
\centering
\subfigure[$\varepsilon=10^{-1}$]{\includegraphics[width=2.6in]{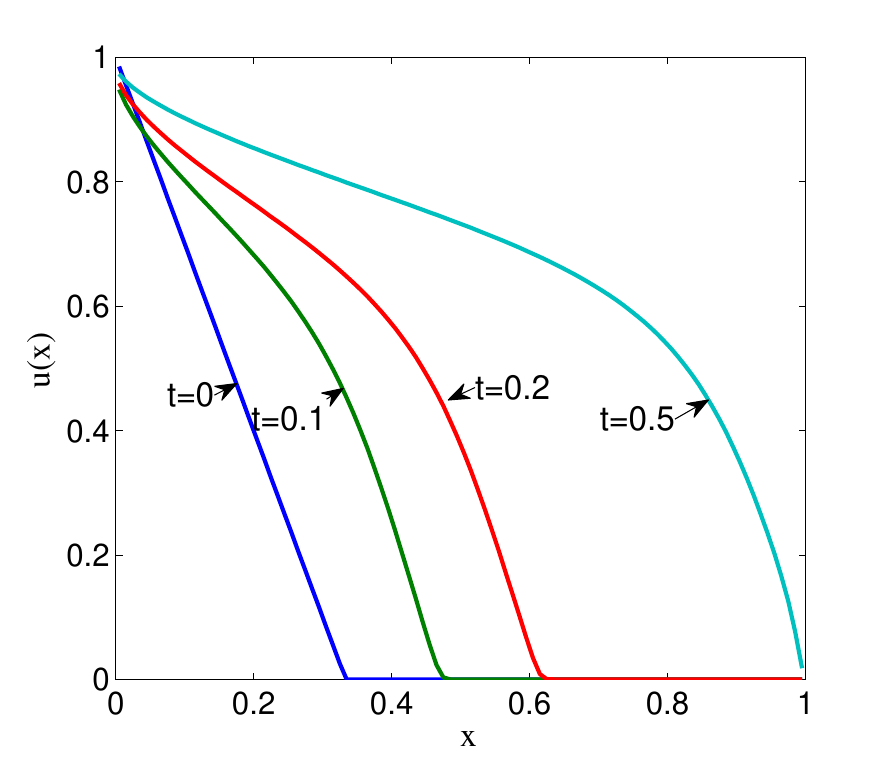}}
\subfigure[$\varepsilon=10^{-2}$]{\includegraphics[width=2.6in]{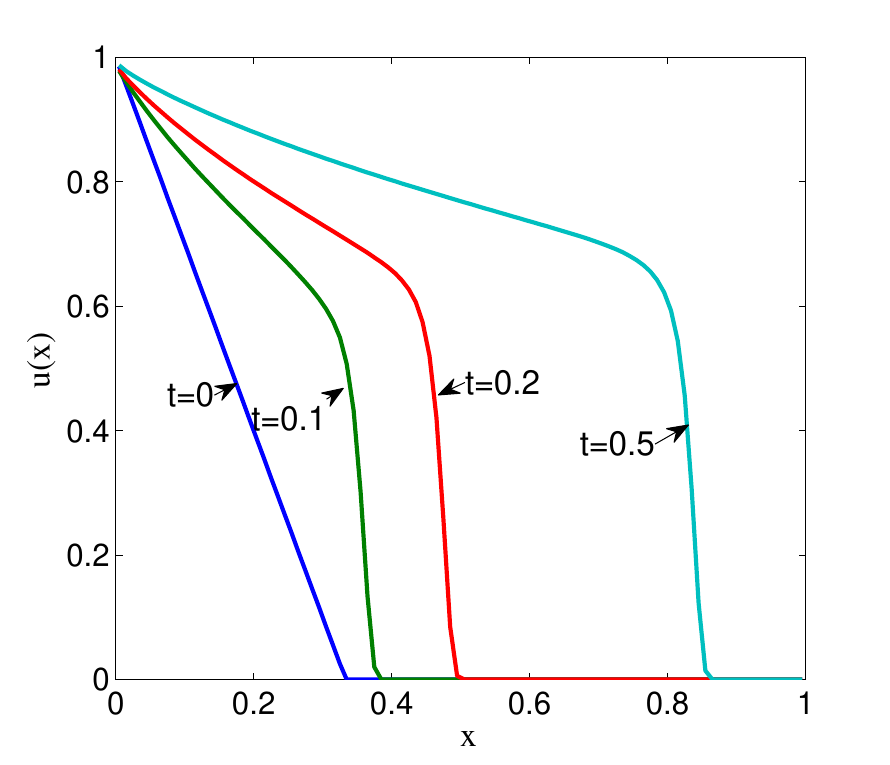}}
\subfigure[$\varepsilon=10^{-3}$]{\includegraphics[width=2.6in]{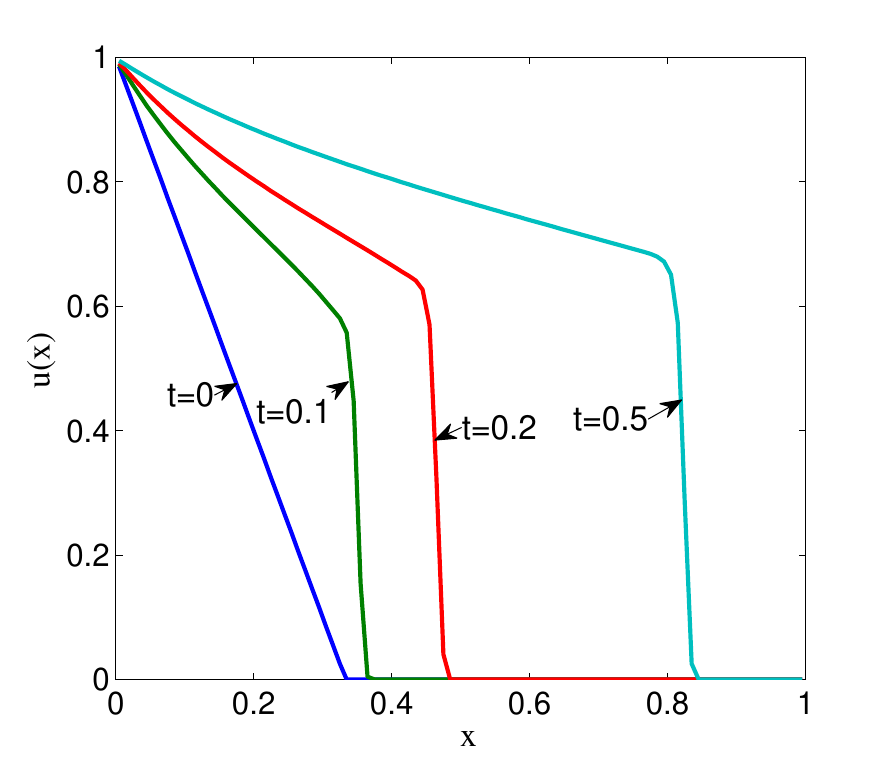}}
\subfigure[$\varepsilon=0$]{\includegraphics[width=2.6in]{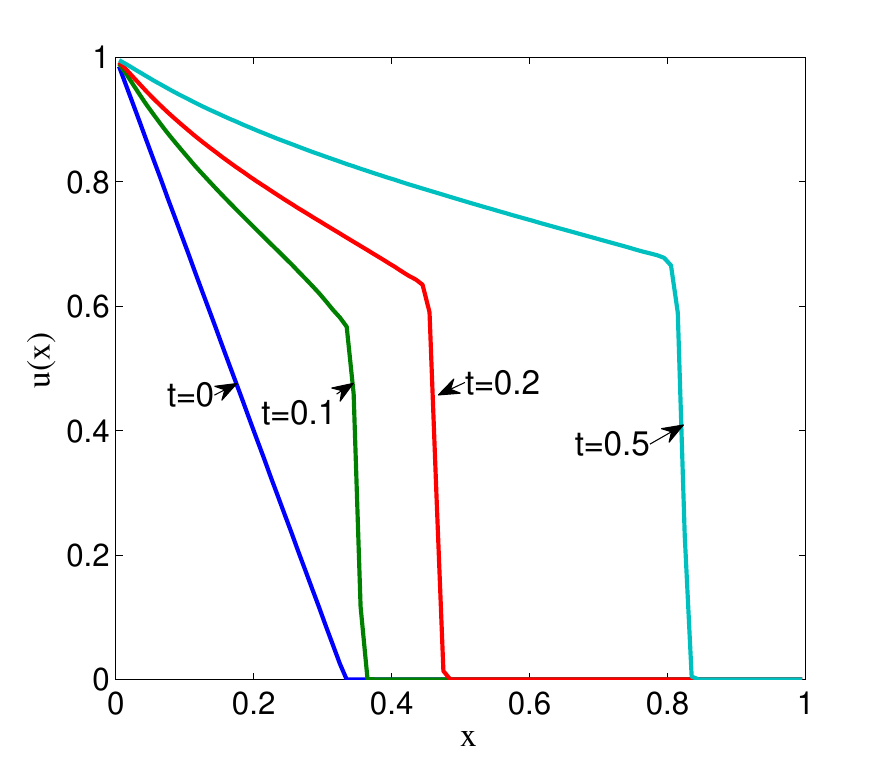}}
\caption{Example 10 - Evolution of the numerical solution for different values of $\varepsilon$.}
\label{ex8_1}
\end{figure}


\section{Conclusion}


In this article we have presented how to build a new finite volume scheme for nonlinear degenerate parabolic equations which admit an entropy functional. To this end, we rewrite the equation in the form of a convection equation, by taking the convective and diffusive parts into account together. Then we apply either the upwind method in the linear case or the local Lax-Friedrichs method in the nonlinear case.\\
On the one hand, this construction ensures that a particular type of steady-state is preserved. We obtain directly a semi-discrete entropy estimate, which is the first step to prove the large-time behavior of the numerical solution. On the other hand, we use a slope-limiter method to get second-order accuracy even in the degenerate case.\\
Numerical examples demonstrate high-order accuracy of the scheme. Moreover we have applied it to some of the physical models for which the long-time behavior has been studied: the porous media equation, the drift-diffusion system for semiconductors, the nonlinear Fokker-Planck equation for bosons and fermions. We obtain the convergence of the approximate solution to an approximation of the equilibrium state at an exponential rate. A future work would be to prove this exponential rate by using a discrete entropy/entropy dissipation estimate as in the continuous case compared with previous approaches.\\

\textbf{Acknowledgement:} This work was partially supported by the European Research Council ERC Starting Grant 2009, project 239983-NuSiKiMo. The authors thank the referee for suggesting the numerical test (Example 3).

\bibliographystyle{plain}
\bibliography{bibliographie}

\end{document}